\documentclass[12pt]{amsart}
\usepackage{amssymb, latexsym, amsthm} 
\usepackage{epsf} 
\usepackage{wasysym}

\usepackage{amsmath}
\usepackage{comment}
\usepackage[all]{xy}
\usepackage{xcolor}

\newtheorem{thm}{Theorem}[section]

\newtheorem{cor}[thm]{Corollary}
\newtheorem{defn}[thm]{Definition}

\newtheorem{exmpl}[thm]{Example}
\newtheorem{lem}[thm]{Lemma}
\newtheorem{prop}[thm]{Proposition}

\newtheorem{rem}[thm]{Remark}

\makeatletter
\renewcommand\subsection{\@startsection{subsection}{3}{\z@}%
	{-2.5ex\@plus -1ex \@minus -.25ex}%
	{1.25ex \@plus .25ex}%
	{\normalfont\normalsize\bfseries}}
\makeatother
\makeatletter
\renewcommand\subsubsection{\@startsection{subsubsection}{3}{\z@}%
	{-2.5ex\@plus -1ex \@minus -.25ex}%
	{1.25ex \@plus .25ex}%
	{\normalfont\normalsize\bfseries}}
\makeatother
\makeatletter
\renewcommand\paragraph{\@startsection{paragraph}{4}{\z@}%
	{-2.5ex\@plus -1ex \@minus -.25ex}%
	{1.25ex \@plus .25ex}%
	{\normalfont\normalsize\bfseries}}
\makeatother
\def\ie{{i.e.}}

\def\C{{\mathbb {C}}}
\def\F{{\mathbb {F}}}

\def\Q{{\mathbb {Q}}}

\def\Z{{\mathbb {Z}}}

\DeclareMathAlphabet{\mathscr}{OT1}{pzc}{m}{it}

\def\ra{{\rightarrow}}

\def\minusset{{-}}

\def\sub{\subseteq}
 
\def\({\left(}
\def\){\right)}

\def\co{{\,{:}\,}}

\newcommand\comp[1]{{#1^{\operatorname{c}}}}

\DeclareMathOperator{\spec}{spec}

\newcommand{\indeg}{{\operatorname{in}\!\!\,\mbox{-}\!\operatorname{deg}}}

\renewcommand{\span}{\operatorname{span}}

\DeclareMathOperator{\mychar}{char}

\newcommand\algint[2][]{\if!#1\relax O_{#2}\else{O_{#2}^{\phantom{I}}}\fi}

\renewcommand\H[4][!]{{\operatorname{H}^{#2}\!\!\;({#3},{#4}{\if!#1\relax\else(#1)\fi})}}

\newcommand\cond[2][!]{{\operatorname{cond}_{\if!#1\relax\else{\comp{#1}}\fi}(#2)}}

\renewcommand\L[2]{{\operatorname{L}^{#1}(#2)}}

\newcommand{\GL}{{\operatorname{GL}}}

\newcommand\SL{{\operatorname{SL}}}

\newcommand\abs[2][]{\left|{#2}\right|_{#1}}

\newcommand{\set}[1]{{\{#1\}}}
\newcommand{\card}[1]{{\left|{#1}\right|}}
\newcommand\ideal[1]{{\left<{#1}\right>}}

\newcommand\sg[1]{{\ideal{#1}}}

\newcommand\db[1]{{(\:\!\!({#1})\:\!\!)}}

\newcommand\innprod[3][]{{\left<{#2},{#3}\right>_{#1}}}

\newcommand\Cref[1]{{Corollary~\ref{#1}}}

\newcommand\Pref[1]{{Proposition~\ref{#1}}}

\newcommand\Tref[1]{{Theorem~\ref{#1}}}

\long\def\half#1\halved{{\footnotesize{#1}}}
\long\def\forget#1\forgotten{}

\newcommand\thisfile{{\jobname.tex}}

\newcommand\dom[2]{{{#1}\backslash{#2}}}

\newcommand\paper[6]{{{#1},\ {\it{#2}},\ {#3}\ {#4},\ {#5},\ ({#6}).}}
\newcommand\book[4]{{{#1},\ {{#2}}{\if!#3!\relax\else{,\ {#3}}\fi}{\if!#4!\relax\else{,\ {#4}}\fi}.}}

\newcommand\dd[4][!]{\abs[#2]{\:\!\det{\if!#1\relax\:\!\!\else(#1)\fi}}^{#3} #4}

\newif\ifXY 
\XYfalse   
\ifXY
\usepackage{xy}
\xyoption{all}
\fi

\newif \iffurther 

\furtherfalse

\definecolor{Red}{rgb}{1,0,0}
\definecolor{Green}{rgb}{0,1,0.5}
\definecolor{Blue}{rgb}{0,0,1}
\definecolor{Yellow}{cmyk}{0,0,1,0}
\definecolor{Fuchsia}{rgb}{0.75,0,0.75}
\definecolor{Orange}{rgb}{1,0.5,0}
\definecolor{Borde}{cmyk}{0,1,0.5,0.5}
\definecolor{Olive}{cmyk}{0,0,0.5,0.5}
\definecolor{BrgtPrpl}{cmyk}{0.25,0.5,0,0}
\definecolor{DrtBl}{cmyk}{1,0,0,0.5}
\definecolor{Pink}{rgb}{01,0.5,1}

\newcommand\PGL{{\operatorname{PGL}}}
\renewcommand\dim{{\operatorname{dim}}}

\newcommand\colored[1]{{\textcolor{Red}{#1}}}

\renewcommand\span{{\operatorname{span}}}

\begin{document}

\title
[Quotients by non-uniform lattices]
{Quotients of buildings by non-uniform lattices}

\author{Orit Sela, Uzi Vishne, Mary Schaps} 


\renewcommand{\subjclassname}{
      \textup{2000} Mathematics Subject Classification}

\date{October 25, 2021}

\maketitle

\begin{abstract}
	We consider quotients of the Bruhat-Tits building associated to the 
	projective linear groups of dimension $d>2$ over the function field 
	$\mathbb F_q(t)$ by a non-uniform lattice $\Gamma$ which is a congruence subgroup in the non-uniform lattice $\Delta = \PGL_{d}(\mathbb F_q[\frac{1}{t}])$.  We determine a fundamental domain and demonstrate that the quotient, while not cofinite, is at least of finite covolume.  We do the case $d=3$ in considerable detail. 
\end{abstract}

\section{Introduction}

The theory of quotients of buildings originated \cite{LPS} with the study of quotients of the Bruhat-Tits tree associated to $G=\PGL_2(\mathbb Q_p)$ by a cocompact lattice, where $\mathbb Q_p$ denotes the field of $p$-adic numbers.  Such quotients were called Ramanujan graphs if their adjacency matrix $A_X$ satisfied a certain  condition on its spectral sequence.  The theory was then extended to other fields $F$ and to higher dimensional complexes \cite{Li1, LSV1,Sa} by considering lattices in $\PGL_d({F})$ for $d>2$, and requiring bounds on the simultaneous spectrum of $d-1$ commuting adjacency matrices. The complexes satifying the bounds called Ramanujan complexes were constructed as quotients of the Bruhat-Tits building $\mathcal B_d$ associated with $\PGL_d(F)$, modulo congruence subgroups of uniform lattices of $\PGL_d(F)$.

Now, let $B$ be the Bruhat-Tits building associated to $\PGL_{d}(F)$ as defined in \S 2.2 and let $\Gamma$ be a congruence subgroup of a
uniform lattice $\Delta$ of $\PGL_{d}(F)$. The quotient $\Gamma \setminus B$
generates a $d$ dimensional complex of higher dimension. Replacing the adjacency
matrix by the simultaneous spectrum of the $d-1$ Hecke operators,
Li (\cite{Li2}), Lubotzky, Samuels and Vishne (\cite{LSV1} and
\cite{LSV2}) gave examples of the quotients of $\dom{\Gamma}{B}$
which are Ramanujan complexes.

Samuels \cite{Sam} extended the theory to the case of
non-uniform lattices $\Delta$ (which means the quotient $\Delta \backslash B$ in infinite). She used representation theory
of locally compact groups to prove the following:
\begin{thm}
	Let $G=\PGL_{d}(\mathbb{F}_{q}\db{t})$, $\Delta =
	\PGL_{d}(\mathbb{F}_{q}[1/t])$ and let $B$ be a Bruhat-Tits
	building associated to $G$, where $d > 2$. Then for any
	finite index subgroup $\Gamma$ of $\Delta$, $\dom{\Gamma}{B}$ is
	not Ramanujan.
\end{thm}
 Samuels considered spherical representations and showed that in the non-uniform case they are not tempered.
The work of Samuels was continued by Orit Sela in her Ph.D. thesis from 2012, \cite{Se}, where she gave a complete description of the fundamental domain and showed that it was not compact, but was of finite volume, which she calculated. This paper is reporting on those results. In the meantime, Hong and Kwon \cite{HK} have published the calculation for the case $d = 3$.

In this paper, we investigate the properties of these quotients of nonuniform lattices and give tools for constructing them. The quotient is not  compact, but it is of finite volume and we will give the formulae for finding the volume.

\section{Buildings} \label{sec:4}

We presume basic familiarity with simplicial complexes and posets. Assume we are given  a poset in which every two elements have a lower bound and for which the set of all elements less than or equal to some element $x$ is isomorphic to the posets of subsets of some finite set under inclusion. Then there is a {\bf 
	canonical simplicial complex} in which the simplicies are linearly ordered subsets. We also assume that the definition of the fundamental group of a simplicial complex is known.

\subsection{Coxeter simplicial complexes}
Let $I$ be a index set and $m$ a map
$$m : I \times I \to \mathbb{Z} \cup \{\infty\}$$
with $m(i,j)\geqslant 2$ when $i \neq j$ and $m(i,i) = 1$.
Let $M$ is the matrix $[m_{ij}]$ where $(m_{ij}=m(i,j))$. The {\bf Coxeter group $(W,S)$ of type $M$} has generating set  $S=\{r_i \mid i \in I\}$ and can be defined by the presentation $W = \langle r_{i} | r_{i}^{2} = (r_{i}r_{j})^{m_{ij}} = 1$ for all $i,j \in I \rangle$.  We assume it to be known that every word is equivalent to a reduced word containing no squares of any generator and that all reduced words equivalent to the same word have the same lengths.
For every $J  \subseteq I$, the subgroup $W_{J} = \langle r_{j} | j \in J \rangle$ is called `standard parabolic subgroup' \cite{Br}\cite{Hum}. 
Reverse inclusion makes the set of standard cosets into a poset satisfying the conditions listed above, and the canonical simplicial complex is called the Tits cone. The Tits cone  has a $0$-simplex given by the standard coset $W_I=W$, and every maximal dimension simplex contains this $0$-complex. We now remove the vertex of the Tits cone, which is to say, we exclude  $W_I$ from the set of cosets. 
The $0$-simplices will be standard cosets $wW_J$ where $J$ has one less element than $I$.

 If $W$ is finite, the complex has dimension
$|S|-1$ and is homeomorphic to a sphere of that dimension.  If $W$ in infinite, the Coxeter complex is contractible. 

\subsection{Bruhat-Tits buildings}

A Bruhat-Tits building is a type of simplicial complex with a high degree of symmetry,
 but we will need to establish more notation in order to define it exactly. We begin with a fixed Coxeter group $(W,S)$. The building will contain subcomplexes called apartments which will be isomorphic to the Coxeter simplex of $(W,S)$. 

The maximal simplices are called  chambers, and two chambers are called adjacent if they have in common a maximal face, also called a {\bf wall}. Within any apartment, there is an action of the Coxeter group $W$ on the simplicial complex, and the two chambers which can be mapped one into the other by $r_i$ are called $i$-adjacent, or neighbors of type $i$.
Two chambers are neighbors of type $i \in I$ if they share a common face of type $I \setminus \set{i}$.
Hence, every element $i \in I$ generates an equivalence relation, in which two chambers are equivalent if there is a set of $i$-adjacent chambers connecting one to the other.

\begin{defn}
	Let $f = i_{1}i_{2} \dots i_{k}$ be a word in the free monoid on $I$, and $r_{f} = r_{i_{1}}r_{i_{2}} \cdots r_{i_{k}} \in W$ be the corresponding element of the Coxeter group.  In a simplicial complex $C$ with an action of a Coxeter group $W$,  a \textbf{gallery} is a finite sequence of chambers $(c_{0}, c_{1}, \dots,c_{k})$ such that $c_{j-1}$ is adjacent to $c_{j}$ for every $1 \le j \le k$. The \textbf{type} of the gallery is said to be $i_{1}i_{2} \dots i_{k}$ if $c_{j-1}$ is $i_{j}$-adjacent to $c_{j}$.
\end{defn}

\begin{exmpl}
The Coxeter group $W$ can be thought of as a trivial simplicial complex with the elements as chambers and  $W$ acting by left multiplication. Then the elements $x,y \in W$ give a gallery $(x,y)$ of type~$f$ if $x^{-1}y = r_{f}$.
\end{exmpl}
\begin{defn}
	For any $J \subset I$,
	The {\bf $J$-residue components} of $C$ are sets of chambers, in which every two chambers of a component are connected with a \textbf{gallery} of $j$-adjacent chambers, where $j \in J$. A $J$-residue component is said to have a {\bf rank } $|J|$ and {\bf corank} $|I|-|J|$. 
\end{defn}
These components define a cell complex of $C$, associating to a residue of co-rank $1$ a vertex ($0$-simplex), to a residue of co-rank $2$ an edge ($1$-simplex) and to residue of co-rank $n$ a $(n-1)$-simplex. The $(|I|-2)$-simplices are called {\bf{panels}}.
\begin{defn}\label{build}

    A {\bf{building for a Coxeter group $W$ of type $M$}} is a simplicial complex $X$  together with a set $\mathcal F$ of Coxeter subcomplexes called the {\bf apartments} of the building, of which $X$ is the union. We denote the pair $(X, \mathcal F)$ by $B$ and require it to satisfy 
	\begin{enumerate}
		\item  There is a function $\delta : B \times B \to W$, such that $\delta(x,y) = r_{f}$ iff $f$ is a reduced word and $x$ and $y$ have a gallery of type $f$. 
		\item Each panel of the chamber system lies on at least two chambers.
	\end{enumerate}
\end{defn}

\begin{defn}
	A simplicial complex $C$ is simply connected if it is connected as a simplicial complex, and for all vertices $c \in C$, the fundamental group $\pi_{1}(C)$ is trivial.
\end{defn}

\begin{defn}
	A simply connected building will be called a {\bf Bruhat-Tits building}.
\end{defn}
\begin{exmpl} Let $F$ be a field and $V$ a vector space over $F$ of dimension $n+1$.  We let $(W,S)$ be the symmetric group permuting $1,2,\dots,n$, with generating set the transpositions $(j, j+1)$. We define a simplicial complex $X$ in which the simplices are chains $V_{i_1} \subset V_{i_2} \subset \dots \subset V_{i_k}$ of subspaces, in which each $V_{i_j}$ is generated by $i_j$ independent elements and is thus of dimension $i_j$.   Chains $V_{1} \subset V_{2} \subset \dots \subset V_{n}$ of subspaces, where $\dim(V_i) = i$, are called maximal flags.
	 
	The maximal flags will be the chambers of the building, $B$, where two maximal sequences, $V_{1} \subset V_{2} \subset \dots \subset V_{n}$ and $V_{1}' \subset V_{2}' \subset \dots \subset V_{n}'$ are $i$-adjacent if $V_{j} = V_{j}'$ for all $j \ne i$. So, the index set $I$ is $\{1, 2, \dots, n \}$. The shared $n-1$ simplex, which is a face of type $I \setminus \{i\}$, is the sequence $V_{1} \subset V_{2} \subset \dots \subset V_{i-1} \subset V_{i+1} \subset \dots \subset V_{n}$.
\end{exmpl}

\subsection{Buildings of type $\tilde{A}_{d}$}

The affine buildings of type $\tilde{A}_{d}$ are buildings, for which the associated Weyl group $W$ has a Dynkin diagram of type $\tilde{A}_d$: a single cycle with $d+1$ vertices. We now give an explicit construction of such buildings.

We will give the construction in positive characteristics, although a similar one can be given in characteristic zero as well. Let $F$ be a local field of characteristic $p$, which we choose to be $F=\F_q((1/t))$, and let $v : F \to \Z$ be  its valuation. Assume $\frac{1}{t}$ is a uniformizer. Let $\mathcal O = \set{x \in F \mid v(x) \geqslant 0}$, and assume the residue field is $\mathcal O/\frac{1}{t}\mathcal O = \mathbb{F}_{q}$, a field with $q = p^{n}$ elements. For each $a \in F^{\times}$, $a \mathcal{O} = (\frac{1}{t})^{v(a)}\mathcal{O}$.

Let $V$ be a vector space over $F$ of dimension $d$. A lattice $L$ of $V$ is any finitely-generated $\mathcal O$-submodule of $V$, which generates $V$ over $F$. Two lattices $L,L'$ are equivalent if $L' = aL$ for some $a \in F^{\times}$. The equivalence class of lattice $L$ is denoted by $[L]$. These equivalence classes will be the vertices of the simplicial complex we are going to build.

The chambers of the building are the sets of vertices corresponding to the equivalence classes of the lattices of the maximal complete sequences $\frac{1}{t}L_{0} \subset L_{k} \subset \dots \subset L_{1} \subset L_{0}$ (that is, the complete flags in the quotient $L_{0}/\frac{1}{t}L_{0}$). Hence, the number of vertices in each chamber is $d$, and two chambers are $i$-adjacent if the associated complete sequences differ in only one lattice, $L_{i}$. For a given choice of equivalence class $[L_0]$ and a choice of standard basis, there are $d!$ different chambers
\[
\frac{1}{t}L_{0} \subset L_{k} \subset \dots \subset L_{1} \subset L_{0},
\]
depending on the order in which the $e_j$ are multiplied by the uniformizer $\frac{1}{t}$, and they make up the orbit of a finite subgroup $\mathring{W}$ of the infinite Weyl group $W$.

The $J$-residue components are the sets of $d - |J| + 1$ vertices, which correspond to the equivalence classes of the lattices of the sequences of order $d - |J| + 1$.
Now, define $B^{i}$ to be the set of vertices associated to the equivalence classes of the lattices $L_{j}, 0 \le j \le i$, which can be ordered in a sequence $\frac{1}{t}L_{0} \subset L_{i} \subset \dots \subset L_{1} \subset L_{0}$. So, $B^{0}$ is the set of vertices, such that every vertex $x$ corresponds to an equivalence class $[L]$; and $B^{1}$ is a set of edges between the vertices of $B^{0}$, such that there is an edge between vertices $[L], [L']$ of $B^{0}$ if there are a lattice $L_{0} \in [L]$ and a lattice $L_{1} \in [L']$ satisfy $\frac{1}{t}L_{0} \subset L_{1}\subset L_{0}$ (The edge is not ordered because if $\frac{1}{t}L_{0} \subseteq L_{1} \subseteq L_{0}$ then $\frac{1}{t}L_{1} \subseteq  \frac{1}{t}L_{0} \subseteq L_{1}$).

By the definition above, the sequences define a complex (which is a building because it is simply connected) and each $B^{i}$ is the set of $i$-simplices of the complex.

The connection to groups stems from the following fact. The group $\GL_{d}(F)$ acts transitively on the lattices of $F^{d}$ and preserves equivalence and inclusion between lattices. The stabilizer of the standard lattice, $L_{0} = {\mathcal O}^{d}$, is the maximal compact subgroup $\GL_{d}(\mathcal O)$; hence, the stabilizer of $L_{0}$ under the induced action of $\PGL_{d}(F)$ is  $\PGL_{d}(\mathcal O)$, which is a maximal compact subgroup by itself.

\begin{exmpl}\label{chambers}
   Let $d=3$, and let $L_0$ be the lattice spanned by $e_1,e_2,e_3$.	Let $c$ be the chamber 
   $\frac{1}{t}L_0 \subset L_2 \subset L_1 \subset L_0$ with 
   \[
   L_1 = \span_{\mathcal O}\{\frac{1}{t}e_1,e_2,e_3\}
   \] 
   \[
   L_2 = \span_{\mathcal O}\{\frac{1}{t}e_1,\frac{1}{t}e_2, e_3\}  
   \] 
   The $1$-adjacent chamber of $c$ would have $L_1$ replaced by $L_1'=
    \span_{\mathcal O}\{e_1,\frac{1}{t}e_2,e_3\}$. The $2$-adjacent chamber would have 
    $L_2$ replaced by $L_2'=\span_{\mathcal O}\{\frac{1}{t}e_1,e_2, \frac{1}{t}e_3\}$. Acting alternately by these two operators, we get all six chambers in the apartment.
\end{exmpl} 
\subsection{Colorings}

 Let $c: \frac{1}{t}L_{0} \subset L_{d-1} \subset \dots \subset L_{0}$ be a chamber, with $L_{0} = \span_{\mathcal O}\{e_{1}, \dots, e_{d}\}$, then $L_{1}$ is $\span_{\mathcal O}\{e_{1}, e_{2}, \dots, \frac{1}{t}e_{i_{1}}, \dots, e_{d}\}$ for some $1 \le i_{1} \le d$ and
$L_{k} = \span_{O}\{e_{1}, \dots , \frac{1}{t}e_{i_{1}},\dots, \frac{1}{t}e_{i_{2}}, \dots ,\frac{1}{t}e_{i_{k}}, \dots, e_{d}\}$, $1 \le i_{1}, \dots, i_{k} \le d$ where $i_{j}$ are different in pairs.
For every lattice $L_{k}$ in the flag, there is a diagonal matrix $g \in \GL_{d}(F)$ such that $gL_{0} = L_k$, and\
\[
g_{ij} = \left\{ \begin{array}{cc}
	1 & \textrm{if } i=j \textrm{ and }  j\not\in \{i_1,i_2, \dots,i_k\}\\	\frac{1}{t} & \textrm{if  } i=j \textrm{ and }  j  \in \{i_1,i_2, \dots,i_k\}\\
	0 & \textrm{otherwise}
\end{array}\right.
\]

\begin{defn}
Let $x=[L]$, $L=gL_0$. The color function of the vertex $x$ is defined to be: $$CL_V(x)=\nu(\det(g))\pmod{d}.$$ 
\end{defn}

	 Let $c: \frac{1}{t}L_{0} \subset L_{d-1} \subset \dots \subset L_{0}$ be a chamber  and
	$L_{k}$ a lattice that belongs to the chamber's chain with correspondent diagonal matrix $g \in \GL_{d}(F)$. The {\bf color function} of a vertex $x \in B^{0}$, which corresponds to a lattice $L_{k}$ is: 
\[	
CL_{V}(x) = \nu(\det(g))\pmod{d} =k.
\]	

\begin{rem}
	The $d$ vertices of the chamber $c$ are of $d$ different colors ($k=0,1,2,\ldots,d-1$).
\end{rem}

\begin{exmpl}\label{coloring}
	Let $d=3$ and let $c$ be the chamber in Example \ref{chambers}. Set 
	\[
	x_0=[L_0],x_1=[L_1], x_2=[L_2]
	\]

	Then 
\[	
CL_V(x_0)=0,\quad CL_V(x_1)=1,\quad CL_V(x_2)=2.
\]
\end{exmpl}
Not only the vertices of the building $B^{0}$ can be colored.  The color of a vertex is defined as above by the place of the vertex in the chamber it belongs to (as the index of the position $[\frac{1}{t^i}L_{0}]$) and the color of an edge between two vertices is defined as the difference of the two colors (modulo $d$).
\begin{defn}
The {\bf color function} of edges 
\[
CL_{E}: B^{0} \times B^{0} \to \Z/d\Z
\] 
is defined to be 
\[
CL_{E}(x,x') = i
\] if 
$\frac{1}{t}L' \subset L \subset L'$ and  \
$[L' : L] =  \frac{1}{t^i}$,
where $L \in [L] = x$ and $L' \in [L'] = x'$. 
\end{defn}
\begin{rem}
	\quad{}
	\begin{enumerate}
		\item
	The number of different colors of edges is $d$ and the edge between $x$ and $x'$ is said to be of color $i$ if $CL_E(x,x')=i$. \\
		\item Denote by $B[i]$ the set of edges of $B$ of color $i$. By definition of colors, if $([L],[L'])$ $\in B[i]$ then $([L'],[L])$ $\in B[d-i]$.
	\end{enumerate}
\end{rem}

The connection between two vertices $L$ and $L'$ with $CL_{E}([L], [L']) = i$ finds expression in the basis of the vertices as a $\mathcal O$-modules.
\begin{exmpl}\label{coloring}
	Let $d=3$ and let $c$ be the chamber in Example \ref{chambers}. Set 
	\[
	x_0=[L_0], \quad x_1=[L_1],\quad x_2=[L_2]
	\]
	Then 
	\[
	CL_E(x_1,x_0)=1,\quad CL_E(x_2,x_0)=2,\quad CL_E(x_2,x_1)=1.
	\]
\end{exmpl}

\begin{rem}
	\quad{}
	\begin{enumerate}
		\item
		If $CL_{E}(x,x') = i$ then there is a path of $i$ edges of color $1$ between $x$ and $x'$.\\
		\item The action of the group $\SL_{d}(F)$ on the building preserves colors.
	\end{enumerate}
\end{rem}

\subsection{Lattices in locally compact groups}

A topological group is {\bf locally compact}
if every non-trivial element belongs to a compact
neigborhood of the identity. This family is preserved
under quotients, i.e., the quotient of a locally compact group is locally compact.

One of the important properties of a locally compact
group $G$ is the fact that it has a Haar measure,
uniquely defined up to multiplication by a scalar, which we will denote here by
$\mu$. The existence of the measure allows to define
the integral of functions on $G$.
We say that the measure is left translation invariant
if for every Borel subset $S \subseteq G$ and for every $a \in G$, $\mu(aS) = \mu(S)$. 
Since the measure is left invariant, namely $\mu(aS)=\mu(S)$ for every $a \in G$ and a Borel subset $S \sub G$ we have for every Borel function $f$ and every $a \in G$ that  $\int_{G}f(ax)d\mu(x) = \int_{G}f(x)d\mu(x)$.

An important example of locally compact groups are the linear groups over archimedean ($\mathbb{R}$ or $\mathbb{C}$) or non-archimedean local field, $F$. The measure is unique up to multiplication by scalar. Let $K$ be a maximal compact subgroup of $G$, then one can normalize the measure such that $\mu(K)=1$.

A lattice, $\Gamma$, of $G$ is a discrete subgroup of finite co-measure, i.e., $\mu(G / \Gamma)< \infty$. There are two kinds of lattices, those for which $G/\Gamma$ is compact and those for which it is not. The former are called uniform lattices and the latter non-uniform lattices.

Define $G = \PGL_{d}(F)$ and $K = \PGL_{d}(\mathcal{O})$. The vertices of the associated building are $B^{0} = G/K$. The measure of all the cosets is the same, because of the invariance of the measure to multiplication by a group element. In particular, the measure of a union of cosets of $K$ is the number of disjoint cosets, so the measure of a subset of $B^{0}$ is the number of the vertices. The measure of a subgroup $A$ containing $K$ is the index $[A:K]$.

Let $B$ be a building of type $\tilde{A}_{n}$ associated to the
local field $\F_{q}\db{\frac{1}{t}}$ An example of a non-uniform lattice is obtained by taking
the lattice to be $\PGL_{d}(\mathbb{F}_{q}[t])$, where the group
is $G=\PGL_{d}(\mathbb{F}_q \db{\frac{1}{t}})$.

In this work,  we use the non-uniform lattice
$\Gamma = \PGL_{d}(\mathbb{F}_{q}[t])$ of the group
$G=\PGL_{d}(\mathbb{F}_{q}\db{ \frac{1}{t}})$ to
describe the fundamental domain of the building $B$
under the action of $\Gamma$.

\section{A quotient of Bruhat-Tits building by a non-uniform lattice} \label{sec:3}

In \cite{Ser2}, Serre gave a very detailed description of the
Bruhat-Tits building of type $\tilde{A}_1$ over a local field $F$,
which is the $(q+1)$-regular tree. A significant portion in this
description is devoted to the quotient with respect to the
non-uniform lattice $\PGL_{2}(\mathcal{O})$, where $\mathcal O$ is the valuation
ring.

We now generalize this description to arbitrary $d$,
giving the  structure of the complex from several viewpoints. We study
the non-uniform lattice $\PGL_{d}(\mathcal{O})$ in detail, with particular
emphasis on the case $d = 3$, which serves as an illustration to
the new phenomenon occurring in higher dimension.

In order to study the quotient, we will find a subset of the building which is a fundamental domain. The following two conditions are needed for proving that an open connected subset $X'$ of the building $B$ is the fundamental domain of $B$ under the action of a group $\Gamma$:
\begin{enumerate}
	\item $B = \bigcup_{\gamma \in \Gamma} \gamma(X')$,
	\item  If $S$ is a simplex in $T$ and then for any $\gamma \in \Gamma$, either $\gamma(S) \bigcap S = \phi$ or $S$ and $\gamma S$ coincide.
\end{enumerate}

\subsection{The distinguished subset $T$ of $B$}

Let $\F_q$ be a finite field, and $F =
\mathbb{F}_{q}\db{\frac{1}{t}}$ the field of Laurant series in the
variable $\frac{1}{t}$. This is a local field, with uniformizer
$\frac{1}{t}$ and valuation ring $\mathcal{O} = \mathbb{F}_{q}[[\frac
{1}{t}]]$. As before, let $K = \PGL_{d}(\mathcal{O})$ be a maximal compact
subgroup of $G = \PGL_{d}(F)$;  $\Gamma =
\PGL_{d}(\mathbb{F}_{q}[t])$ is a non-uniform lattice, and we
point out that $\Gamma\cap K = \PGL_d(\F_q)$, since $\mathbb{F}_{q}[t] \cap \mathcal{O} = \mathbb{F}_{q}$. 

The group $G$ acts on the space of lattices in $F^d$ (which are
$\mathcal O$-submodules containing bases of $F^d$). One can define the
Bruhat-Tits building associated to $\PGL_{d}(F)$ by identifying
the set of vertices $B^{0}$ with the coset space $G/K$. By the
notation $B[i]$, we mean an edge of color $i$.

We will define a distinguished set $T$ of vertices of $B^{0}$ by special properties of their basis
as $\mathcal O$-modules. Later, we will prove that this set $T$ is  a fundamental
domain of B, isomorphic to  $\dom{\Gamma}{B}$, and this gives a set of representatives for the elements of the quotient we are trying to calculate. 
\begin{defn}
The sequence
$\bar{n} = (n_{1}, n_{2}, \cdots, n_{d-1}, 0)$, where \\
$n_{1} \ge n_{2} \ge
\cdots \ge n_{d} = 0$, will be denoted by $\bar{n}$ and it defines
the lattice $l_{\bar{n}}$ and the matrix $L_{\bar{n}}$:
\[
l_{\bar{n}} =                     \sum_{i=1}^{d}t^{n_{i}}\mathcal O e_{i}
\]
where
 $\{e_{i}\}_{i=1}^{d}$ is the standard basis of the columns space
of $F^d$. 
\\
\[
(L_{\bar{n}})_{ij} = 
\left\{ \begin{array}{cc} t^{n_{i}}\mathcal O & \textrm{ if } i=j\\ 0 & \textrm{ otherwise } \\	
\end{array}\right.	
\]
\color{black}

The set of all such vertices in the building $B^0$ will be 
denoted by $T$ and the set of all such sequence $\bar{n}$, will be denoted by $N_T^d$.
\end{defn}

\begin{rem} From module theory of principal ideal domains we know:
	\quad{}
	\begin{enumerate}
		\item If $L$ is a diagonal matrix with columns $\{ t^{n_{i}}e_{i}\}$ and $$\bar{n}=(n_1,n_2, \cdots, n_{d-1},0),\quad n_i \in \mathbb{N} \cup \{0 \},\quad 1 \leq i \leq {d-1}$$,
		then it equivalent with respect to elementary row and columns operations over $\mathcal O$ to a diagonal matrix
		with columns $\{t^{n_{i_{j}}}e_{j}\}$ such that $n_{i_{j+1}} \le
		n_{i_{j}}$ for every $1 \le j \le d-1$ (which means that $(n_{i_1},n_{i_2}, \cdots, n_{i_{d-1}}, 0)$ is a permutation of the $d-1$ elements of $(n_1,n_2,\cdots,n_{d-1},0)$).
		\item
		Every diagonal lattice can be written as $AL_{\bar{n}}A'$, where
		$A,A'$ are in $\PGL_{d}(\mathcal O)$ and $\bar{n} \in N^d_T$.
	\end{enumerate}
\end{rem}

\begin{defn}
Two matrices $L_{\overline{m}}$ and $L_{\overline{n}}$ are equivalent if there exist two matrices $A,A' \in {\PGL_{d}(\mathcal
{O})}$ such that \\ 
\[
L_{\overline{m}}={AL_{\bar{n}}A'}. 
\] \\
We will denote the set of all the matrices, which are 
equivalent to $L_{\bar{n}}$ by $\Lambda_{\bar{n}}=[L_{\bar{n}}]$.
\end{defn}

\begin{defn}
  An oriented edge of color $i$ from $\Lambda_{\bar{m}}$ to $\Lambda_{\bar{n}}$ exists if $\bar{m}=\bar{n} + \bar{v}$, where $\bar{v}=(v_{1}, v_{2}, \cdots, v_{d-1}, 0)$ has $(-1)$'s in $i$ places and $0$'s in $d-i$ places.
\end{defn}
\begin{cor}\label{neighbors}

An oriented edge from vertex
$\Lambda_{\bar{m}}$ to vertex
$\Lambda_{\bar{n}}$, where $\bar{m}=\bar{n} + \bar{v}$, has color $1$ if there is a path of lattices
of length $1$ from $\Lambda_{\bar{m}}$ to
$\Lambda_{\bar{n}}$, 
\[
\Lambda_{\bar{m}} = \Lambda_{1} \subset
\Lambda_{0} = \Lambda_{\bar{n}}
\]
and there is a path of lattices
of length $(d-1)$ from $\frac{1}{t}\Lambda_{\bar{n}}$ to
$\Lambda_{\bar{m}}$, 
\[
\frac{1}{t}\Lambda_{\bar{n}} = \Lambda_{d} \subset
\Lambda_{d-1} \subset \dots \subset \Lambda_{1} = \Lambda_{\bar{m}}
\]
such that
for every $0 \le i \le d$, there is  $L_{\bar{m}} \in \Lambda_{\bar{m}}$ and there is $L_{\bar{n}} \in \Lambda_{\bar{n}}$, such that 
$[L_{\bar{m}} : L_{\bar{n}}] =  \frac{1}{t}$. \\

An oriented edge from vertex
$\Lambda_{\bar{m}}$ to vertex
$\Lambda_{\bar{n}}$, where $\bar{m}=\bar{n} + \bar{v}$, has color $i$ if there is a path of lattices
of length $i$ from $\Lambda_{\bar{m}}$ to
$\Lambda_{\bar{n}}$, 

\[
\Lambda_{\bar{m}} = \Lambda_{i} \subset
\Lambda_{i-1} \subset \dots \subset \Lambda_{0} = \Lambda_{\bar{n}}
\]
and there is a path of lattices
of length $(d-i)$ from $\frac{1}{t}\Lambda_{\bar{n}}$ to
$\Lambda_{\bar{m}}$, 

\[
\frac{1}{t}\Lambda_{\bar{n}} = \Lambda_{d} \subset
\Lambda_{d-1} \subset \dots \subset \Lambda_{i} = \Lambda_{\bar{m}}
\]
such that
for every $0 \le i \le d$ there is an oriented edge of color $1$
between $\Lambda_{i+1}$ and $\Lambda_{i}$.
\end{cor} 
\begin{exmpl}
	Assume $d = 3$. Denote by $n_{1}n_{2}n_{3}$ the equivalence class of $L_{ \bar{n}}$, where $\bar{n} = (n_{1}, n_{2}, n_{3})$. 
 The edges of color $1$, $( \Lambda_{\bar{n}}, \Lambda_{\bar{m}})$, are the diagonal edges from the right up vertex to the left down vertex or vertical edges from down to up or horizontal edges from left to right.\\
 The edges of color $2$, $( \Lambda_{\bar{n}}, \Lambda_{\bar{m}})$, are the diagonal edges from the left down vertex to the right up vertex or vertical edges from up to down or horizontal edges from right to left.\\

	An oriented edge of color $2$, 
    $\Lambda_{\bar{n}} \stackrel{2}{\longrightarrow} \Lambda_{\bar{m}}$, exists if and only if the reverse edge is of color $1$, $\Lambda_{\bar{m}}\stackrel{1}{\longrightarrow}
	\Lambda_{\bar{n}}$. 
 The left-most corner of $T$ is presented in
	Figure~\ref{fig1}.
	\begin{figure}[!h]
	\begin{displaymath}
		\xymatrix{
			& & & & 440 
              \ar[dl]_{1} \ar[r]_{1} & \dots\\
			& & & 330 \ar[dl]_{1} \ar[r]_{1} & 430 \ar[dl]_{1} \ar[r]_{1} \ar[u]_{1} & \dots \\
			& & 220 \ar[dl]_{1} \ar[r]_{1} & 320 \ar[dl]_{1} \ar[u]_{1} \ar[r]_{1} & 420 \ar[dl]_{1} \ar[u]_{1} \ar[r]_{1} & \dots \\
			& 110 \ar[dl]_{1} \ar[r]_{1}  & 210 \ar[dl]_{1} \ar[u]_{1} \ar[r]_{1} & 310 \ar[dl]_{1} \ar[u]_{1} \ar[r]_{1} & 410 \ar[dl]_{1} \ar[u]_{1} \ar[r]_{1} & \dots \\
			000 \ar[r]_{1} & 100 \ar[r]_{1} \ar[u]_{1} & 200 \ar[r]_{1} \ar[u]_{1} & 300 \ar[r]_{1} \ar[u]_{1} & 400 \ar[r]_{1} \ar[u]_{1} & \dots }
	\end{displaymath}
	\caption{Diagram of $T$}\label{fig1}
\end{figure}
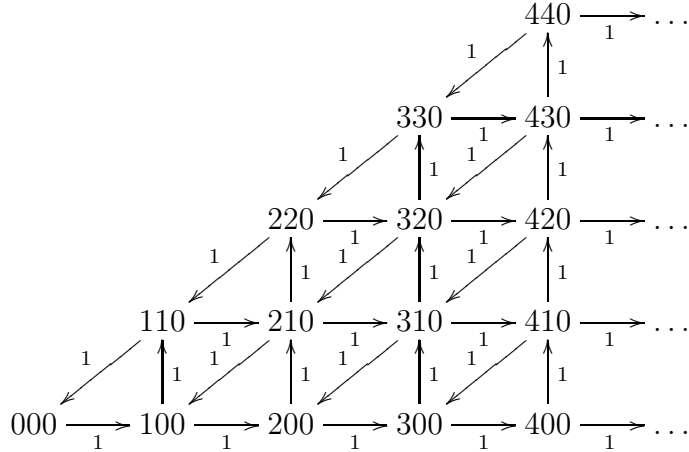

\end{exmpl}

\color{black}
We remark that the analogous graph for $d = 2$, as computed in
\cite{Ser2}, is the one-sided infinite path.

The main step in showing that $T$ is a fundamental domain of
$\Gamma$ is to verify that every element in $B^0$ can be brought to
$T$ by multiplication by elements of $\Gamma$. We prove this by
induction on the distance from $B$, which requires acting not by
the full group $\Gamma$, but by stabilizers of vertices, in order
to preserve the closest vertex which is already in $B$. 

\subsection{The neighbors of a vertex of $B^{0}$}

The vertices which have a common edge of color $i$ with
$\Lambda_{\bar{n}}$ are called neighbors of degree $i$.
The
neighbors of a vertex of $T$ can be in $T$ themselves or outside
of $T$. The number of neighbors depends on the series $\bar{n}$.

In this section, we find the neighbors of every vertex in $T$.
To save space (especially for $d = 3$), we introduce a somewhat
shorter notation $\bar{m}( \bar{n})$: \\
rather than writing
\[ 
\bar{n} = (n_{1}, n_{2}, \dots, n_{d}) 
\] 
we consider the vector   
\[ 
\bar{m}( \bar{n}) = (n_{1}- n_{2}, n_{2} - n_{3}, \dots, n_{d-1} - n_{d})=(m_{1},
m_{2}, \dots, m_{d-1})
\] 
which is called {\bf the sequence of differences of $\bar{n}$} and is a vector of positive integers. This new indexing is useful in dealing with the
stabilizers of a lattice. Of course, if $\bar{m}( \bar{n})$ is
given, then \[n_{i} = \sum_{j=i}^{d-1} m_{j}\] (where $n_{d} = 0$).
We define the integer $|\bar{m}|$ by $|\bar{m}| = | \{ m_{i} : m_{i} \neq 0 \} |$.
Given a vertex $(n_1,\dots,n_d)$ in $T$, if there are $r$ different integers define a sequence $\bar{d}=(d_1 ,d_2, \dots, d_r)$ of $d$ to blocks $d = \sum_{i=1}^{r}d_{i}$, such that for every 
$1 \le i \le r$, $d_{i}$ is the number of equal successive values
$n_j$. For example, if $d = 5$, $r = 3$ and $d_{1} = 2, d_{2} = 1,
d_{3} = 2$, then the block $1$ is $n_{1}, n_{2}$, the block $2$ is
$n_{3}$ and the block $3$ is $n_{4}, n_{5}$, where $n_{1} = n_{2}
> n_{3} > n_{4} = n_{5}$.\\ 
For each $1 \le l \le r$ define
\[
j_{l} = 1+ \sum_{i=1}^{l-1}d_{i} \\
\]
Set $g_{l} = n_{j_{l}}$, then $j_{l}$ is the first index for which $g_{l}$ appears in $\bar{n}$.
\begin{rem}
The set $\bar{n}$ satisfies for every $1 \le j \le d$: \\
\[
n_j = g_l  
\]
where $ 1+ \sum_{i=1}^{l-1}d_{i} \le j \le  \sum_{i=1}^{l}d_{i}$.\\
Moreover, for every $1\le k \le d-1 $:
\[
m_k = \left\{ \begin{array}{ccc}
	g_l -g_{l+1} &  1 + \sum_{i=1}^{l-1}d_{i}
	\le k \le \sum_{i=1}^{l}d_{i} < k+1 \le \sum_{i=1}^{l+1}d_{i}\\	0 & 1 + \sum_{i=1}^{l-1}d_{i}
	\le k,k+1 \le \sum_{i=1}^{l}d_{i}\\
\end{array}\right.
\]     
\end{rem}


\begin{defn}
    Given a vertex $(n_1,\dots,n_d)$ in $T$, with block sequence $\bar{d}=(d_1,d_2,\dots,d_r)$,
$d = \sum_{i=1}^{r}d_{i}$. 
A sequence $\bar{v}$ {\bf preserves $T$} if $\bar{v} = (v_{1}, v_{2},\dots, v_{d})$ with zeros in $d - k$ places and $(-1)$'s in $k$
places, with the condition that if  $0$ appears in
place $j_{0}$ which satisfies
\[
1 + \sum_{m=1}^{i-1}d_{m} \le j_{0} \le \sum_{m=1}^{i}d_{m}
\]
then all
the places $1 + \sum_{m=1}^{i-1}d_{m} \le j < j_{0}$ are also $0$. 

\end{defn}
\begin{lem}
Let $(n_1,\dots,n_d)$ be a vertex in $T$, with block sequence $(d_1,d_2,\dots,d_r)$,
$d = \sum_{i=1}^{r}d_{i}$. 
Let  $\bar{v} = (v_{1}, v_{2},\dots, v_{d})$ be a sequence which preserves $T$, with zeros in $d - k$ places and $(-1)$'s in $k$
places, 
then the neighbors of degree $k$ of the vertex are 
 $[L^{
	\bar{v}}] = \Lambda^{\bar{v}}$ with columns
\begin{displaymath} (L^{ \bar{v}})_{j} = t^{n_{j}+v_{j}}e_{j}
. \end{displaymath}	
where $\{e_j\}, 1\le j\le d$, are the elements of the standard basis of $F^d$.\\
Let the series $\bar{n}$ and $\bar{v}$ be as above with the
condition that if $k = 1$ then $v_{d} = 0$. 
The neighbors of degree $k$ which are not in $T$ are $[L^{ \bar{v}}]$
with columns of the form:
\begin{displaymath} (L^{ \bar{v}})_{ij} = \left\{ \begin{array}{ccc}
		t^{n_{i}+v_{i}}\mathcal{O} & i=j \\
  t^{n_{i}}\mathcal{O} & j=i+1 \\
  0 & \textrm{     otherwise} \\
	\end{array} \right. \end{displaymath}

\end{lem}

\begin{proof}
 Let
$\Lambda_{\bar{n}}$ be a vertex of $T$ and $\bar{d}( \bar{n}) =
(d_{1},d_{2}, \cdots ,d_{r})$ be the block sequence.
By Corollary \ref{neighbors}, the neighbors of the vertex of degree $k$, $1 \le k \le d-1$,
that belong to $T$ are as follows. Suppose $\bar{v} = (v_{1}, v_{2},
\dots, v_{d})$ preserves $T$ with zeros in $d - k$ places and $(-1)$'s in $k$
places, 
then all the places $ 1 + \sum_{m=0}^{i-1}d_{m} \le j < j_{0}$ are also $0$
($d_{0} = 0$), else, the neighbor will not be in $T$, because the sequence $ \bar{n}$ does not satisfying the condition that $n_0 \ge n_1 \ge \dots \ge n_d$. 

A lattice $L$ is a neighbor of the lattice $L_{\bar{n}}$ if $L$ satisfies $\frac{1}{t}L_{\bar{n}} \subset L \subset L_{\bar{n}}$, so the powers of each line $j$ of matrix $L$ are $n_j$ or $n_{j}-1$. By column operations we can consider only the upper triangular matrices.
We denote by $\bar{v}^{j}$, $1 \leq j \leq k$ the $k$ vectors satisfying the condition that for each $2 \leq j \leq k$ the vector $\bar{v}^{j} - \bar{v}^{j-1}$ is a vector with $d-1$ zeros and one $1$ in place $l$. If $[L]$ is a neighbor of degree $k$, then there is a sequence of containment $\frac{1}{t}L_{\bar{n}} \subset L = L^{\bar{v}^{1}} \subset L^{\bar{v}^{2}} \subset \dots \subset L^{\bar{v}^{k}}  \subset L_{\bar{n}}$. If there is a $t^{n_l}$ in place $(l,l+1)$, then it can became a zero by a column operation using the diagonal element $t^{n_l}$ in row $l$. The matrices in the sequence are of the form:

\begin{displaymath} (L^{ \bar{v}^{l}})_{ij} = \left\{ \begin{array}{ccc}
		t^{n_{i}+v_{i}}\mathcal{O} & i=j \\
  t^{n_{i}}\mathcal{O} & j=i+1, v_{i} = -1 \\
  0 & \textrm{     otherwise} \\
	\end{array} \right. \end{displaymath}

\end{proof}
\color{black}
\begin{exmpl} \label{expml1} We assume $d = 3$, and describe the neighbors
	of $\bar{n} = (2, 1, 0)$ in $B$. The goal is to show that each
	neighbor is equivalent to a neighbor in $T$ under the action of
	the stabilizer of $\Lambda_{\bar{n}}$ in $\Gamma$. In other words, every
	edge $(\Lambda_{\bar{n}},\Lambda_{\bar{n'}})$, $\Lambda_{\bar{n'}} \in B\setminus T$, is equivalent under
	$\Gamma$ to an edge of the same form with $(\Lambda_{\bar{n}},\Lambda_{\bar{m}}), \Lambda_{\bar{m}} \in T$.
	
	For this particular vertex, $d_{1} = d_{2} = d_{3} = 1$, and $r =
	3$.  The neighbors of $\Lambda_{ \bar{n}} = \left \{ \left(
	\begin{array}{ccc}
		\alpha_1 t^{2} & 0 & 0\\
		0 & \alpha_2 t & 0\\
		0 & 0 & \alpha_3 
	\end{array} \right) :\alpha_{i} \neq 0, \alpha_{i} \in \mathcal{O} \right \}$ are all lattices $[L]$ that satisfy:\\
	\begin{displaymath} \left \{ \left( \begin{array}{ccc}
		\alpha_1 t & 0 & 0\\
		0 & \alpha_2  & 0 \\
		0 & 0 & \alpha_3 t^{-1}
	\end{array} \right): \alpha_{i} \in \mathcal{O} \right \} \subseteq L \subseteq \left \{ \left( \begin{array}{ccc}
		\beta_1 t^{2} & 0 & 0\\
		0 & \beta_2 t & 0\\
		0 & 0 &\beta_3  
	\end{array} \right): \beta_{j} \in \mathcal{O} \right \} \end{displaymath}
	
	Denote the set of neighbors in $T$ and of
	degree $j$ by $I^{j}=\{I^{j}_{i}\}$. The neighbors not in $T$
	and of degree $j$ will be denoted by $O^{j}=\{O^{j}_{i}\}$.

	The neighbors of degree $k = 1$ in $T$ are:
	\begin{displaymath}
		\bar{v} = (0, 0, -1),
		\textrm{ } I_{1}^{1} = L^{ \bar{v}} = \left \{ \left(
		\begin{array}{ccc}
			\alpha_1 t^{2} & 0 & 0 \\
			0 & \alpha_2 t & 0 \\
			0 & 0 & \alpha_3 t^{-1}
		\end{array} \right): \alpha_{i} \in \mathcal{O} \right \} \approx 21(-1)
		\approx 320;
	\end{displaymath}
	\begin{displaymath}
		\bar{v} = (0, -1, 0), \textrm{ } I_{2}^{1} = L^{ \bar{v}} = \left \{\left(
		\begin{array}{ccc}
			\alpha_1 t^{2} & 0 & 0 \\
			0 & \alpha_2 & 0 \\
			0 & 0 & \alpha_3
		\end{array} \right): \alpha_{i} \in \mathcal{O} \right \} \approx 200;
	\end{displaymath}
	\begin{displaymath}	
	\bar{v} = (-1, 0, 0), \textrm{ } I_{3}^{1} = L^{ \bar{v}} = \left \{\left(
		\begin{array}{ccc}
			\alpha_1 t & 0 & 0 \\
			0 & \alpha_2 t & 0 \\
			0 & 0 & \alpha_3
		\end{array} \right): \alpha_{i} \in \mathcal{O} \right \} \approx
		110;
	\end{displaymath}
	\\
The neighbors of degree $k = 2$ in $T$ are:
	\begin{displaymath}
		\bar{v} = (0, -1, -1), 
  \textrm{ } I_{1}^{2} = L^{ \bar{v}} = \left \{ \left( \begin{array}{ccc}
			\alpha_1 t^{2} & 0 & 0\\
			0 & \alpha_2 & 0 \\
			0 & 0 & \alpha_3 t^{-1}
		\end{array} \right): \alpha_{i} \in \mathcal{O} \right \} \approx 20(-1) \approx
		310;
	\end{displaymath}
	\begin{displaymath}
		\bar{v} = (-1, 0, -1), \textrm{ } I_{2}^{2} = L^{ \bar{v}} = \left \{\left(
		\begin{array}{ccc}
			\alpha_1 t & 0 & 0\\
			0 & \alpha_2 t & 0 \\
			0 & 0 & \alpha_3 t^{-1}
		\end{array} \right): \alpha_{i} \in \mathcal{O} \right \} \approx 11(-1) \approx
		220;
	\end{displaymath}
	\begin{displaymath}
	 \bar{v} = (-1, -1, 0), \textrm{ } I_{3}^{2} = L^{ \bar{v}} = \left \{ \left(
		\begin{array}{ccc}
			\alpha_1 t & 0 & 0\\
			0 & \alpha_2 & 0\\
			0 & 0 & \alpha_3
		\end{array} \right): \alpha_{i} \in \mathcal{O} \right \} \approx 100.
	\end{displaymath}
	\\

	The neighbors of degree $k = 1$ outside of $T$ are:
 
	\begin{displaymath}
		\bar{v} = (0, -1, 0), \textrm{ } O_{1}^{1} = L^{ \bar{v}} = 
  \left \{\left( \begin{array}{ccc}
			\alpha_1 t^{2} & \beta_1 t^{2} & 0 \\
			0 & \alpha_2 & \beta_2 t\\
			0 & 0 & \alpha_3
		\end{array} \right): \alpha_{i}, \beta_{j}\in \mathcal{O} \right \} 
	\end{displaymath}
  \\ 
	\begin{displaymath}  
  \approx 
  \left 
  \{\left( \begin{array}{ccc}
			\alpha_1 t^{2} & 0 & 0 \\
			0 & \alpha_2  & \beta t \\
			0 & 0 & \alpha_3
		\end{array}
		\right): \alpha_{i}, \beta \in \mathcal{O} \right \};
	\end{displaymath}
 
	\begin{displaymath}
	\bar{v} = (-1, 0, 0), \textrm{ } O_{2}^{1} = L^{ \bar{v}} =
		\left \{\left(
		\begin{array}{ccc}
			\alpha_1 t & \beta_1 t^{2} & 0 \\
			0 &\alpha_2 t & \beta_{2}t \\
			0 & 0 & \alpha_3
		\end{array} \right):\alpha_{i},  \beta_{j} \in \mathcal{O} \right \} 
  \\
\end{displaymath}
\begin{displaymath}	  
  \approx 
  \left \{\left( \begin{array}{ccc}
			\alpha_1 t & \beta_1 t^{2} & 0 \\
			0 & \alpha_2 t & \beta_{2}t \\
			0 & 0 & \alpha_3
		\end{array}
		\right): \alpha_{i}, \beta_{i} \in \mathcal{O} \right \}.
	\end{displaymath}

 The neighbors of degree $k = 2$ outside of $T$ are:
	\begin{displaymath}
		\bar{v} = (-1, 0, -1), \textrm{ } O_{1}^{2} = L^{ \bar{v}} =\left \{ \left(
		\begin{array}{ccc}
			\alpha_1 t & \beta_1 t^{2} & 0 \\
			0 & \alpha_2 t & \beta_2 t \\
			0 & 0 & \alpha_3 t^{-1}
		\end{array}
		\right): \alpha_{i}, \beta_{j} \in \mathcal{O} \right \} ;
	\end{displaymath}
	\begin{displaymath}
	\bar{v} = (-1, -1, 0), \textrm{ } O_{2}^{2} = L^{ \bar{v}} =
		\left \{\left(
		\begin{array}{ccc}
			\alpha_1 t & \beta_1 t^{2} & 0 \\
			0 & \alpha_2 & \beta_2 t \\
			0 & 0 & \alpha_3 
		\end{array} \right): \alpha_{i}, \beta_{j} \in \mathcal{O} \right \}.
	\end{displaymath}
	\\
\begin{displaymath}
		\bar{v} = (0, -1, -1), \textrm{ } O_{3}^{2} = L^{ \bar{v}} =\left \{ \left(
		\begin{array}{ccc}
			\alpha_1 t^2 & \beta_1 t^{2} & 0 \\
			0 & \alpha_2  & \beta_2 t \\
			0 & 0 & \alpha_3 t^{-1}
		\end{array}
		\right): \alpha_{i}, \beta_{j} \in \mathcal{O} \right \} ;
	\end{displaymath}

	If $L_{\bar{n}}$ has a neighbor in other form, then it equivalent
	to a neighbor of the above form over $ \mathcal{O}$ by column operations.
\end{exmpl}

\begin{exmpl} \label{ex:ngb}
	
	Again assume $d = 3$ and consider the vertex $\bar{n} = (2,1,0)$.
	In this example, we will show a neighbor of $\Lambda_{210}$ with $\bar{n}=(n_{1},n_{2}, n_{3})$
	for which it is not true that $n_{1} \ge n_{2} \ge n_{3}$, but has an equivalent form satisfying this condition.
	
	The neighbor $\left( \begin{array}{ccc}
		t^{3} & 0 & 0\\
		\alpha t^{2} & t & 0\\
		\beta t & 0 & 1
	\end{array} \right)$ can become the defining neighbor
	$\left( \begin{array}{ccc}
		t^{2} & 0 & \alpha't^{3}\\
		0 & t & \beta't^{2}\\
		0 & 0 & t
	\end{array} \right)$.\\
	Denote by $C_{1}$, $C_{2}$, $C_{3}$ the columns of the above first matrix. Then:\\
	\begin{displaymath}
		\left( \begin{array}{ccc}
			t^{3} & 0 & 0\\
			\alpha t^{2} & t & 0\\
			\beta t & 0 & 1
		\end{array} \right) 
        \stackrel{C_3 \ra C_3+C_1}{\rightarrow}
		\left( \begin{array}{ccc}
			t^{3} & 0 & t^{3}\\
			\alpha t^{2} & t & \alpha t^{2}\\
			\beta t & 0 & 1+ \beta t
		\end{array} \right)
		\stackrel{C_{3} \ra C_{3} - \frac{1}{ \beta t}C_{1}}{\rightarrow}
	\end{displaymath}
	\begin{displaymath}
		\left( \begin{array}{ccc}
			t^{3} & 0 & t^{3}-\frac{t^{2}}{\beta}\\
			\alpha t^{2} & t & \alpha t^{2} - \frac{ \alpha}{ \beta}t\\
			\beta t & 0 & \beta t
		\end{array} \right) \stackrel{C_{1} \to C_{1} - C_{3}}{\rightarrow}
		\left( \begin{array}{ccc}
			t^{3}-t^{3}+ \frac{t^{2}}{ \beta} & 0 & t^{3}-\frac{t^{2}}{\beta}\\
			\alpha t^{2}- \alpha t^{2} + \frac{ \alpha}{ \beta}t & t & \alpha t^{2} -                                               \frac{ \alpha}{ \beta}t\\
			\beta t- \beta t & 0 & \beta t
		\end{array} \right)
	\end{displaymath}
	\begin{displaymath}
		=
		\left( \begin{array}{ccc}
			\frac{t^{2}}{ \beta} & 0 & t^{3}-\frac{t^{2}}{\beta}\\
			\frac{ \alpha}{ \beta}t & t & \alpha t^{2} -   \frac{ \alpha}{ \beta}t\\
			0 & 0 & \beta t
		\end{array} \right) \stackrel{C_{1} \to \beta C_{1}}{\rightarrow}
		\left( \begin{array}{ccc}
			t^{2} & 0 & t^{3}-\frac{t^{2}}{\beta}\\
			\alpha t & t & \alpha t^{2} - \frac{ \alpha}{ \beta}t\\
			0 & 0 & \beta t
		\end{array} \right)
		\stackrel{C_{1} \to C_{1} - \alpha C_{2}}{\rightarrow}
	\end{displaymath}
	\begin{displaymath}
		\left( \begin{array}{ccc}
			t^{2} & 0 & t^{3}-\frac{t^{2}}{\beta}\\
			0 & t & \alpha t^{2} -  \frac{ \alpha}{ \beta}t\\
			0 & 0 & \beta t
		\end{array} \right) \stackrel{C_{3} \to C_{3} + \frac{1}{ \beta} C_{1}}{\rightarrow}
		\left( \begin{array}{ccc}
			t^{2} & 0 & t^{3}\\
			0 & t & \alpha t^{2} -\frac{ \alpha}{ \beta}t\\
			0 & 0 & \beta t
		\end{array} \right) \stackrel{C_{3} \to C_{3} + \frac{ \alpha}{ \beta} C_{2}}{\rightarrow}
	\end{displaymath}
	\begin{displaymath}
		\left( \begin{array}{ccc}
			t^{2} & 0 & t^{3}\\
			0 & t & \alpha t^{2}\\
			0 & 0 & \beta t
		\end{array} \right)
		\stackrel{C_{3} \to \frac{1}{ \beta}C_{3}}{\rightarrow}
		\left( \begin{array}{ccc}
			t^{2} & 0 & \frac{1}{ \beta}t^{3}\\
			0 & t & \frac{\alpha}{ \beta} t^{2}\\
			0 & 0 & t
		\end{array} \right).
	\end{displaymath}
\end{exmpl}
\begin{defn}
We call $\bar{c}=(c_{1},c_{2},\dots,c_{d-1})$, $c_{i} \in \{+1, -1, 0 \}$, $i=1, 2, \dots, d-1$, a {\bf chain}. We call $\bar{c}$ an {\bf alternating chain}
if there are no two consecutive plus or minus signs(zeros can
appear between them). 
\end{defn}

\begin{lem}
The neighbors $\bar{n}^{2}$ of a vertex $\bar{n}^{1}$ in $T$ are those vertices
which satisfy:

\begin{enumerate}
	\item $\bar{m}^{2}$ is $\bar{m}^{1}$ with changes $\bar{c} = \bar{m}^{2} - \bar{m}^{1}$, which are an alternating chain
	of length $1 \le j \le d$. \\
	\item If $m_{j}^{1} = 0$ then $c_{j}$ can not be minus. \\
\end{enumerate}
\end{lem}
\begin{proof}	
	Let 
	$\bar{n}^{1}=(n_{1}^{1},n_{2}^{1},...,n_{d}^{1}), \bar{n}^{2}=(n_{1}^{2},n_{2}^{2},...,n_{d}^{2})$ be two neighbors in $T$, such that
	$\bar{n}^{2}=\bar{n}^{1} + \bar{v}$ or $\bar{v}=\bar{n}^{2} \minusset \bar{n}^{1} = (n_{1}^{2}\minusset n_{1}^{1},...,n_{d}^{2} \minusset n_{d}^{1})$.\\
	 For each vertex we can define $\bar{m}^{1}=(n_{1}^{1}-n_{2}^{1},n_{2}^{1}-n_{3}^{1},...,n_{d-1}^{1}-n_{d}^{1}), \bar{m}^{2}=(n_{1}^{2}-n_{2}^{2},n_{2}^{2}-n_{3}^{2},...,n_{d-1}^{2}-n_{d}^{2})$ respectively, such that \\
	 $\bar{c}=\bar{m}^{2}-\bar{m}^{1}=(n_{1}^{2}-n_{2}^{2}-n_{1}^{1}+n_{2}^{1},n_{2}^{2}-n_{3}^{2}-n_{2}^{1}+n_{3}^{1},...,n_{d-1}^{2}-n_{d-1}^{1})=(v_{1}-v_{2},v_{2}-v_{3},...,v_{d-1}-v_{d})$, $v_{i} \in \{{0,-1} \}, 1 \le i \le d-1$.
	 That is $\bar{c}=(c_{1},c_{2},...,c_{d-1})$ satisfies $c_{i} \in \{{0,\pm{1}}\}$.\\
	 We will prove that if for some $1\le j \le d-1$, $m_{j}=0$, then $c_{j} \ge 0$. \\
	 Let $m_{j}=0$, if $c_{j}=-1$, then $v_{j}-v_{j+1}=-1$ and $v_{j}=-1, v_{j+1}=0$, which means that $n_{j}^{2}=n_{j}^{1}-1, n_{j+1}^{2}=n_{j+1}^{1},n_{j}^{1}=n_{j+1}^{1}$, since $m_j = 0$ that leads to $n_{j}^{2}=n_{j}^{1}-1,n_{j+1}^{2}=n_{j}^{1}$ or that $n_{j}^{2} < n_{j+1}^{2}$ in contrast to the definition of a vertex in $T$.\\
	 Now remains to be proven that there is no sequence of two $1$ or sequence of two $(-1)$ in $\bar{c}$.
	 Suppose there is $1 \le j \le d-1$ such that $c_{j}=c_{j+1}=1$ (or $c_{j}=c_{j+1}=-1$), then $v_{j+2}=v_{j}-2$ (or $v_{j+2}=v_{j}+2$), but $v_{j},v_{j+2} \in \{0, -1\}$.
\end{proof}
\begin{lem}
	Let $\bar{n} \in N^{d}_{T}$ with sequence of differences $\bar{m}$.\\ The number of neighbors (in $B$) of color $1$ is 
	$d$.\\
 The number of neighbors in $T$ of color $1$ is 
	$ (1+|m|)$.
\end{lem}
\begin{proof}
We deifne the integer $|\bar{m}|$ by $|\bar{m}| = | \{ m_{i} : m_{i} \neq 0 \} |$.
In that case, the sequence is of length $2$, begins with $+$ without zeros between the plus
and minus or it can be a specific sequence of length~$1$, which has $+1$ in place $c_{1}$.
\end{proof}

\subsection{The stabilizer of a vertex}

In this section, we compute for every vertex $\Lambda \in T$ its
stabilizer under the action of $\Gamma$. We determine the action of the stabilizer on the neighbors,
finding the fixed points and the other orbits. This is later used to show that
$T$ is the fundamental domain of $\Gamma$.

\begin{rem}
Let $\Lambda_{\bar{n}}$ be a vertex of $T$. The stabilizer
$\Gamma_{ \bar{n}}$ of
$\Lambda_{\bar{n}}$ under the action of $\Gamma$ on $B$  are those
elements $\gamma \in \Gamma$ such that $\gamma \cdot [L_{ \bar{n}}]
= [ L_{ \bar{n}}]$. In terms of matrices, this is equivalent to
$\gamma L_{\bar{n}}  = L_{\bar{n}}\cdot o$ for some $o \in
\PGL_{d}(\mathcal O)$.
\end{rem}
The stabilizers are structured as a block upper diagonal matrix. This structure depends on $\bar{n}$.
Let $\bar{d} = (d_{1},d_{2},\dots,d_{l})$ be a block sequence, where the
block $l_{i}$ has the value $g_{l_{i}}$. Let $M_{ \bar{n}}$ be a
matrix of $\PGL_{d}( \mathbb{F}_{q}[t])$ with blocks of size
$d_{l} \times d_{m}$, $1 \le l, m \le r$. For every $i$ and $j$
there are $l$ and $m$, such that $\sum_{k=1}^{l}d_{k}-d_{l} < i \le
\sum_{k=1}^{l}d_{k}$ and  $\sum_{k=1}^{m}d_{k}-d_{m} < j \le
\sum_{k=1}^{m}d_{k}$. If $M$ is a matrix in the stabilizer of $\bar{n}$, then 

\begin{displaymath} M_{ij} = \left\{ \begin{array}{cc}
		0 & l > m \\
  f^{g_{l}-g_{m}}_{ij} & l \le m \\
	\end{array} \right. \end{displaymath}
and $deg(f^{g_{l}-g_{m}}_{ij}) \le g_{l}-g_{m}$.
\begin{exmpl}
	Take $d = 6$. If $\bar{n} = (n_{1}, n_{2}, n_{3}, n_{4}, n_{5},
	n_{6}) = (6, 4, 4, 4, 0, 0)$, then $r = 3$ and $d_{1} = 1$, $d_{2}
	= 3$, $d_{3} = 2$, where $n_{1}$ is in $d_{1}$, $n_{2} = n_{3} =
	n_{4}$ are in $d_{2}$ and $n_{5} = n_{6}$ are in $d_{3}$.
	
	The matrix of the differences, $A = (a_{ij})$, $a_{ij} = n_{i} -
	n_{j}$ is  $$\left( \begin{array}{cccccc}
		0 & 2 & 2 & 2 & 6 & 6\\
		-2 & 0 & 0 & 0 & 4 & 4 \\
		-2 & 0 & 0 & 0 & 4 & 4\\
		-2 & 0 & 0 & 0 & 4 & 4\\
		-6 & -4 & -4 & -4 & 0 & 0\\
		-6 & -4 & -4 & -4 & 0 & 0
	\end{array} \right).$$
	
	Therefore, the stabilizer is composed of matrices $M_{\bar{n}}$ of
	the form $$\left(
	\begin{array}{cccccc}
		f^{0}_{11} & f^{2}_{12} & f^{2}_{13} & f^{2}_{14} & f^{6}_{15} & f^{6}_{16}\\
		0 & f^{0}_{22} & f^{0}_{23} & f^{0}_{24} & f^{4}_{25} & f^{4}_{26} \\
		0 & f^{0}_{32} & f^{0}_{33} & f^{0}_{34} & f^{4}_{35} & f^{4}_{36}\\
		0 & f^{0}_{42} & f^{0}_{43} & f^{0}_{44} & f^{4}_{45} & f^{4}_{46}\\
		0 & 0 & 0 & 0 & f^{0}_{55} & f^{0}_{56}\\
		0 & 0 & 0 & 0 & f^{0}_{65} & f^{0}_{66}
	\end{array} \right),$$\\
	where for every $1 \le i, j \le 6$ $f_{ij}^{k} \in \mathbb{F}_{q}[t]$,
	and $\deg(f^{k}_{ij}) \le k$.
\end{exmpl}

\begin{prop}
	Let $\Lambda_{\bar{n}}$ be a vertex of $T$
	and denote by $\Gamma_{\bar{n}}$
	the stabilizer of the vertex. Then $\Gamma_{ \bar{n}} = \{M_{\bar{n}}\}$.
\end{prop}

	
\begin{proof}
	The matrix $L_{ \bar{n}}$ has the form $(t)_{ij}$, where $(t)_{ij} = \left\{ \begin{array}{ll} t^{n_{i}} & i=j \\
		0 &    i \ne j
	\end{array} \right.$ and $1 \le i, j \le d$.
	
	Let $\gamma$ be an element of $\Gamma_{ \bar{n}}$. The equation $\gamma L_{ \bar{n}} = L_{ \bar{n}}o$ ($ o \in \PGL_{d}(\mathcal{O}$) leads to the equation $(t_{ij})^{-1} \gamma (t_{ij}) = o$ and then $t^{-n_{i}} \gamma_{ij}t^{n_{j}} = o_{ij} \in \mathcal{O}$.
	Use the valuation of the field, $v(o_{ij}) = v(t^{-n_{i}} \gamma_{ij}t^{n_{j}}) \ge 0$ and $v(t^{-n_{i}} \gamma_{ij}t^{n_{j}}) = n_{i} - n_{j} + v( \gamma_{ij}) \ge 0$, $v( \gamma_{ij}) \le 0 $. \\ The conclusion of this is that $0 \ge v( \gamma_{ij}) \ge n_{j} - n_{i}$.\\
	So, $\gamma_{ij} = \left \{ \begin{array}{ll}
		c \in F & i = j\\
		g \in F[t], deg(g) \le n_{i}-n_{j} & i < j \\
		0 & j < i
	\end{array} \right.$ and $\gamma$ is of the form $M_{\bar{n}}$.

\end{proof}

Let $(\Lambda_{\bar{n}^{1}}, \Lambda_{\bar{n}^{2}})$
be an edge of $T$ between vertices $\Lambda_{\bar{n}^{1}}$ and
$\Lambda_{\bar{n}^{2}}$. Denote by $\Gamma(\bar{n}^{1}, \bar{n}^{2})$
the stabilizers of the edge $(\Lambda_{\bar{n}^{1}}, \Lambda_{\bar{n}^{2}})$,
which is defined by
$\Gamma(\bar{n}^{1}, \bar{n}^{2}) = \Gamma_{\bar{n}^{1}} \cap \Gamma_{\bar{n}^{2}}$ and this is a subgroup of $\Gamma_{\bar{n}}$, hence a subgroup of $\Gamma$.

{\bf The order of $\Gamma_{\bar{n}}$.}

The number of possible polynomials in a block $d_{l} \times d_{m}$,
where $l < m$, which is polynomial of degree less or equals to  $g_{l} - g_{m}$ is 
\[q^{d_{l}d_{m}(g_{l} - g_{m} + 1)} = q^{d_{l}d_{m}(
	\sum_{k=l}^{d-1}m_{k} - \sum_{k=m}^{d-1}m_{k} +1)} =
q^{d_{l}d_{m}( \sum_{k=l}^{m-1}m_{k} +1)}\]
If $l = m$ then the
polynomials of block $d_{l} \times d_{l}$ has degree $0$, but the
block has to be invertible, so the number of possible such
polynomials is $|\GL_{d_{l}}(F)|$. If $l > m$ then the block
$d_{l}  \times d_{m}$ is a block of zeros. Hence,
$$|\Gamma_{ \bar{n}}| =
\frac{1}{q-1} \prod_{i=1}^{r}|\GL_{d_{i}}(F)|q^{ \sum_{i<j \le
		r}d_{i}d_{j}( \sum_{l=i}^{j-1}m_{l}+1)}.$$

\begin{rem}
	It is worth noting that for some neighbors, $u,v$, the stabilizer
	of $u$ contains the stabilizer of $v$. More precisely, $\Gamma_{
		\bar{n^{1}}} \subseteq \Gamma_{ \bar{n^{2}}}$ iff:
	\begin{itemize}
		\item The series $\bar{n^{1}}$ and $\bar{n^{2}}$ have the
		same partition of $d$, \item
		The series $\bar{m} ( \bar{n^{1}}) =
		(m_{1}^{1}, m_{2}^{1}, \dots, m_{d-1}^{1})$ and \\
		$\bar{m} (\bar{n^{2}} ) = (m_{1}^{2}, m_{2}^{2} \dots, m_{d-1}^{2})$ satisfy
		$m_{i}^{1} \le m_{i}^{2}$ for all $1 \le i \le d-1$.
	\end{itemize}
\end{rem}
\begin{proof}
    We know that $ m_k=g_l-g_{l+1}$ or $m_k=0$, then if $g_l^1-g_{l+1}^1 \le g_l^2-g_{l+1}^2$ this leads to $m_k^1 \le m_k^2$.
\end{proof}
If $\Gamma_{ \bar{m}} \subseteq \Gamma_{ \bar{n}}$ then $\Gamma_{
	\bar{m}}$ trivially stabilizes $\Lambda_{ \bar{n}}$.
\begin{rem}
	If $\bar{n}^{2}$ is a neighbor of $\bar{n}^{1}$, then $|m_{i}^{1}
	- m_{i}^{2}| \le 1$ for every $1 \le i \le d-1$.
\end{rem}
\color{black}
\begin{exmpl}
	Let $d=3$. The diagram of the inclusions between the stabilizers
	is given in the following diagram, where $\Gamma_{ \bar{n}}$
	contains $\Gamma_{ \bar{m}}$ if and only if there is an arrow from
	$\Gamma_{ \bar{m}}$ to  $\Gamma_{ \bar{n}}$:
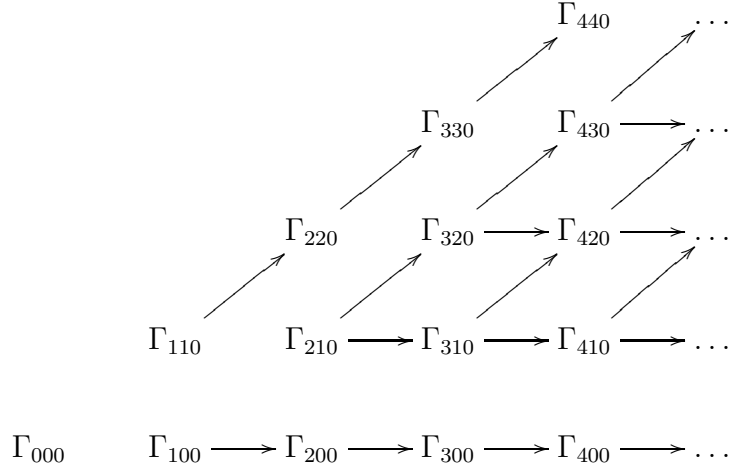
	\begin{displaymath}
		\xymatrix{
			& & & & \Gamma_{440} & \dots\\
			& & & \Gamma_{330} \ar[ur] & \Gamma_{430} \ar[ur] \ar[r] & \dots \\
			& & \Gamma_{220} \ar[ur] & \Gamma_{320} \ar[ur] \ar[r] & \Gamma_{420} \ar[ur] \ar[r] & \dots \\
			& \Gamma_{110} \ar[ur] & \Gamma_{210} \ar[ur] \ar[r] & \Gamma_{310} \ar[ur] \ar[r] & \Gamma_{410} \ar[ur] \ar[r] & \dots \\
			\Gamma_{000} & \Gamma_{100} \ar[r] & \Gamma_{200} \ar[r] & \Gamma_{300} \ar[r] & \Gamma_{400} \ar[r] & \dots }
	\end{displaymath}
\end{exmpl}

\subsection{The action of the stabilizer on other
	vertices}
There are three kinds of neighbors for a vertex $\Lambda_{
	\bar{n}} \in T$, those which belong to $B \backslash T$, those which are in $T$ and the relations between the stabilizers are of containing (the stabilizer of the vertex $\Lambda_{\bar{n}}$ is contained in the stabilizer of the neighbor), those which are in $T$, but do not have a relation of
containment. 
\begin{defn}
Let $\Lambda_{\bar{n}}$ be a vertex in $T$.
    The neighbors of $\Lambda_{ \bar{n}}$ in $T$ that are stabilized by $\Gamma_{
	\bar{n}}$ will be called its {\bf friends}.
\end{defn}

\begin{lem}
	Let $ \bar{n}$ be an index of a vertex in $T$ and $\bar{m} =
	(m_{1}, m_{2}, \dots, m_{d-1})$ suitable to $\bar{n}$.
	Then there are $|\bar{m}|$ friends of $\bar{n}$ in $T$ under the
	action of $\Gamma_{\bar{n}}$.
\end{lem}
The only case, in which a vertex does not have friends, is
when $\bar{n} = \bar{0}$.
\begin{lem}
	Let $\Lambda_{ \bar{n}}$ be a vertex of $T$. Then $\Gamma_{
		\bar{n}}$ is the stabilizer of the set of all neighbors of
	$\Lambda_{ \bar{n}}$.
\end{lem}
\begin{proof}
	If $[L]$ is a neighbor of $\Lambda_{ \bar{n}}$ then
	$\frac{1}{t}L_{ \bar{n}} \subseteq L \subseteq L_{ \bar{n}}$, so
	the action of $\gamma \in \Gamma_{ \bar{n}}$ is $\gamma
	\frac{1}{t}L_{ \bar{n}} \subseteq \gamma L \subseteq \gamma L_{
		\bar{n}} \approx L_{ \bar{n}}$.
\end{proof}
\begin{exmpl}
    Let $\bar{n}^1 =(3,3,0)$. The neighbors of color 1 are $\bar{n}^2=(4,4,0), \bar{n}^3 = (3,2,0)$. Then $d^1 _1 =2, d^1_2 =1, d^2_1 =2, d^2_2 = 1, d^3_1 =1, d^3_2 = 1, d^3_3 =1$ and $\bar{m}^1 = (0,3), \bar{m}^2 = (0,4), \bar{m}^3 = (1,2)$, so $\bar{n}^2$ is a friend of $\bar{n}^1$ and $\bar{n}^3$ is not a friend of $\bar{n}^1$.\\
    Let $\bar{n}^1 =(3,2,0)$. The neighbors of color 1 are $\bar{n}^2=(4,3,0), \bar{n}^3 = (3,1,0), \bar{n}^4 =(2,2,0)$. Then $d^1 _1 =1, d^1_2 =1, d^1_3=1, d^2_1 =1, d^2_2 = 1, d^2_3=1, d^3_1 =1, d^3_2 = 1, d^3_3 =1, d^4_1=2, d^4_2=1$ and $\bar{m}^1 = (1,2), \bar{m}^2 = (1,3), \bar{m}^3 = (2,1), \bar{m}^4 =(0,2)$, so $\bar{n}^4$ and $\bar{n}^3$ are not a friends of $\bar{n}^1$, but $\bar{n}^2$ is a friend of color $1$ of $\bar{n}^1$.
\end{exmpl}

\begin{rem}
	Let $\bar{n} = (n_{1},n_{2},\dots, n_{d})$ such that $m_{d-1} \ne 0$. 
 Then there is only one friend of degree $1$ of a vertex
	$[L_{\bar{n}}]$ in $T$. 
     The friend is $[L_{\bar{n}}^1]$, 
	$\bar{n}^1 = (n^1_{1},n^1_{2},\dots, n^1_{d-1}, n^1_{d}-1)$   
    (otherwise, the condition $\bar{m}_{i}^{1} \le \bar{m}_{i}^{2}$ is not preserved
	for every $1 \le i \le (d-1)$). If $m_{d-1} = 0$ then there is no
	friend of degree $1$. \\
 
 In general, if $\bar{n} = (n_{1},n_{2},\dots, n_{d})$ with block sequence $\bar{d} = (d_{1},d_{2}, \dots, d_{r})$ has a friend of degree $k$ in $T$ iff there exist $i_0$, such that $k = d_1 + d_2 + \cdots +d_{i_0}$. The friend is $\bar{n}^k = (n_1, n_2, \dots, n_{d-k}, n_{d -k+1}-1, \dots , n_d -1)$ .
 
\end{rem}

\begin{lem}
	Let $\Lambda_{ \bar{n}}$ be a vertex of $T$, where $\bar{n} =
	(n_{1}, \dots, n_{d}=0)$ and suppose  $m_{d-1} \ne 0$.
	$\Gamma_{\bar{n}}$ moves the
	neighbors of degree $p$ of $[L_{ \bar{n}}]$ in $T$, which are not
	its friends, to neighbors of degree $p$ of $[L_{ \bar{n}}]$
	outside of $T$, and vice versa. The changes in the matrix of the
	suitable neighbor will be in the same columns. If $m_{d-1} = 0$ then
	$\Gamma_{\bar{n}}$ moves the neighbors of degree $p$ of
	$[L_{\bar{n}}]$ in $T$, which are not its friends, to
	neighbors of degree $p$ of $[L_{ \bar{n}}]$ outside of $T$  and
	vice 
 versa, and the changes in the matrix of the suitable neighbor
	will be in the same blocks.
\end{lem}

\begin{proof}
	We will explain the process for $p = 1$, and the other cases can be
	proved in the same way. Let $\bar{n}$ be $(n_{1}, n_{2}, \dots,
	n_{d})$ and $ \bar{d} = (d_{1},d_{2}, \cdots, d_{r})$ its block sequence. Denote by $[L_{ \bar{n}}^{1}]$ a
	neighbor of $[L_{\bar{n}}]$ of degree $1$ in $T$, which is not a 
	friend. $L_{\bar{n}}^{1}$ has the form 
 
 $\sum_{j_{0} \ne i =
		1}^{d}t^{n_{i}-1}e_{i} + t^{n_{j_{0}}}e_{j_{0}}$, for some $j_{0}=
	\sum_{k=0}^{l}d_{k} + 1$, $0 \le l \le r - 1$ and $d_{0} = 0$. For
	example: If $d=5$, $\bar{n} = 44210$ and $\bar{n}^{1} = 44310$. 
 Let $\gamma
	\in \Gamma_{ \bar{n}}$, the matrix $L = \gamma L_{ \bar{n}}^{1}$
	has the form: $(L)_{ij} = \sum_{k=1}^{d}
	\gamma_{ik}(L_{\bar{n}}^{1})_{kj}$. If $j = j_{0}$ then
	$(L)_{ij_{0}} = \gamma_{ij_{0}}t^{n_{j_{0}}}$, otherwise,
	$(L)_{ij} = \gamma_{ij}t^{n_{j}-1}$. The degree of $(L)_{ij}$ will
	be: if $i = j_{0}$ then $\deg((L)_{ij_{0}})= deg( \gamma_{ij_{0}})
	+ n_{j_{0}} = n_{j_{0}}$. Otherwise, $\deg((L)_{ij}) = deg(
	\gamma_{ij}) + n_{j} - 1 = n_{i} - n_{j} + n_{j} - 1 = n_{i} - 1$.\\
	\begin{exmpl}
 For $d=5$, let $\bar{n} = 44210$ and $\bar{n}^{1} = 44110$.
	We denote by $C_{i}$ for $1 \le i \le 5$ the columns of the matrix, then: \\
	$$\gamma L_{\bar{n}}^{1} = \left( \begin{array}{ccccc}
		f_{11}^{4} & f_{12}^{4} & f_{13}^{3} & f_{14}^{4} & f_{15}^{4}\\
		f_{21}^{4} & f_{22}^{4} & f_{23}^{3} & f_{24}^{4} & f_{25}^{4} \\
		0 & 0 & f_{33}^{1} & f_{34}^{2} & f_{35}^{2}\\
		0 & 0 & 0 & f_{44}^{1} & f_{45}^{1}\\
		0 & 0 & 0 & 0 & f_{55}^{0}
	\end{array} \right).$$\\   
	\end{exmpl}

 Let $[L_{\bar{n}}^{2}]$ be a neighbor of $[L_{ \bar{n}}]$ of degree $1$ outside of $T$: \\
 \[
 (L_{\bar{n}}^2)_{ij}= \left\{ \begin{array}{cc}
 t^{n_i} & i=j \neq j_{0} \\
 t^{n_{i}-1} & i=j_{0} \\
 t^{n_i} & j=i+1, j_{0} < j \leq d\\
 0 & \textrm{otherwise} \\
 \end{array}
 \]

 For example, a neighbor of degree 1 outside of $T$:
 $$ L_{\bar{n}}^{2} = \left( \begin{array}{ccccc}
		f_{11}^{4} & f_{12}^{4} & f_{13}^{3} & f_{14}^{4} & f_{15}^{4}\\
		f_{21}^{4} & f_{22}^{4} & f_{23}^{3} & f_{24}^{4} & f_{25}^{4} \\
		0 & 0 & f_{33}^{1} & f_{34}^{2} & f_{35}^{2}\\
		0 & 0 & 0 & f_{44}^{1} & f_{45}^{1}\\
		0 & 0 & 0 & 0 & f_{55}^{0}
	\end{array} \right).$$\\  
	
	Now, we can act on matrix $L$ (in order to get a matrix in $B \setminus T$)
	by changing columns. The places, which are not zeros (but they are scalars
	of $\mathbb{F}_{q}$) and they lie under the diagonal, can become zeros
	by the columns of the block they belong to (right to them). \\
	For example: 
	\[
	 \left( \begin{array}{ccccc}
		f_{11}^{4} & f_{12}^{4} & f_{13}^{3} & f_{14}^{4} & f_{15}^{4}\\
		f_{21}^{4} & f_{22}^{4} & f_{23}^{3} & f_{24}^{4} & f_{25}^{4} \\
		0 & 0 & f_{33}^{1} & f_{34}^{2} & f_{35}^{2}\\
		0 & 0 & 0 & f_{44}^{1} & f_{45}^{1}\\
		0 & 0 & 0 & 0 & f_{55}^{0}
	\end{array} \right) \stackrel{C_{1} \to C_{1} - \alpha C_{2}}{\longrightarrow} \left( \begin{array}{ccccc}
		f_{11}^{4} & f_{12}^{4} & f_{13}^{3} & f_{14}^{4} & f_{15}^{4}\\
		0 & f_{22}^{4} & f_{23}^{3} & f_{24}^{4} & f_{25}^{4} \\
		0 & 0 & f_{33}^{1} & f_{34}^{2} & f_{35}^{2}\\
		0 & 0 & 0 & f_{44}^{1} & f_{45}^{1}\\
		0 & 0 & 0 & 0 & f_{55}^{0}
	\end{array} \right).
\]
	The elements above the diagonal can became zeros by the diagonal
	elements except column $j_{0}$. By this action, $L$ becomes a
	neighbor of degree $1$ in $B \setminus T$.
	For example:
	\begin{displaymath}
		\left(
		\begin{array}{ccccc}
			f_{11}^{4} & 0 & f_{13}^{3} & f_{14}^{4} & f_{15}^{4}\\
			0 & f_{22}^{4} & f_{23}^{3} & f_{24}^{4} & f_{25}^{4} \\
			0 & 0 & f_{33}^{1} & f_{34}^{2} & f_{35}^{2}\\
			0 & 0 & 0 & f_{44}^{1} & f_{45}^{1}\\
			0 & 0 & 0 & 0 & f_{55}^{0}
		\end{array} \right)  \stackrel{\begin{array}{ccc}
				C_{2} \to C_{2} - \alpha C_{1}\\
				C_{3} \to C_{3} - \frac{1}{t} C_{1}\\
				C_{4} \to C_{4} - \alpha C_{1}\\
             C_{5} \to C_{5} - \alpha C_{1}
		\end{array}}{\longrightarrow}
		\left( \begin{array}{ccccc}
			f_{11}^{4} & 0 & 0 & 0 & 0\\
			0 & f_{22}^{4} & f_{23}^{3} & f_{24}^{4} & f_{25}^{4} \\
			0 & 0 & f_{33}^{3} & f_{34}^{2} & f_{35}^{2}\\
			0 & 0 & 0 & f_{44}^{1} & f_{45}^{1}\\
			0 & 0 & 0 & 0 & f_{55}^{0}
		\end{array} \right) \\
\end{displaymath}
\begin{displaymath} 
 \stackrel{\begin{array}{ccc}
				C_{3} \to C_{3} - \frac{1}{t} C_{1}\\
				C_{4} \to C_{4} - \alpha C_{1}\\
				C_{5} \to C_{5} - \alpha C_{1}\\
                C_{5} \to C_{5} - \alpha C_{4}
		\end{array}}{\longrightarrow} 
		\left( \begin{array}{ccccc}
			f_{11}^{4} & 0 & 0 & 0 & 0\\
			0 & f_{22}^{4} & 0 & 0 & 0 \\
			0 & 0 & f_{33}^{1} & f_{34}^{2} & 0\\
			0 & 0 & 0 & f_{44}^{1} & f_{45}^{1}\\
			0 & 0 & 0 & 0 & f_{55}^{0}
		\end{array} \right)
	\end{displaymath}
	\\
	In the example below, it can be shown that the elements $L_{i3}$
	for $i < 3$ cannot became zeros. The neighbor
	$[L_{\bar{n}}^{2}]$ which is not in $T$ can become a neighbor of
	degree $1$ in $T$ by using $\gamma^{-1}$.
\end{proof}
\begin{exmpl}
	Assume $\bar{n}=210$ as in Example~\ref{expml1}. Then $\gamma \in \Gamma_{ \bar{n}}$ is of
	the form $\left( \begin{array}{ccc}                                   1 & f_{12}^{1} & f_{13}^{2} \\
	0 & 1 & f_{23}^{1} \\
	0 & 0 & 1
	\end{array} \right)$,
	so the action of $\gamma$ on the neighbors is (the signs are as in example~\ref{expml1}:\\
	\begin{displaymath}
	\gamma I_{2}^{1} = \left( \begin{array}{ccc}                                   1 & f_{12}^{1} & f_{13}^{2} \\
	0 & 1 & f_{23}^{1} \\
	0 & 0 & 1
	\end{array} \right) \left( \begin{array}{ccc}                                   t^2 & 0 & 0 \\
	0 & 1 & 0 \\
	0 & 0 & 1
	\end{array} \right) = \left( \begin{array}{ccc}                                   f_{11}^{2} & f_{12}^{1} & f_{2} \\
	0 & f_{22}^{0} & f_{23}^{1} \\
	0 & 0 & f_{33}^{0}
	\end{array} \right) = I_{2}^{1}
 \end{displaymath}
 \begin{displaymath}
     \gamma I_{3}^{1} =\left( \begin{array}{ccc}                                   1 & f_{11}^{1} & f_{13}^{2} \\
	0 & 1 & f_{23}^{1} \\
	0 & 0 & 1
	\end{array} \right) \left( \begin{array}{ccc}                                   t & 0 & 0 \\
	0 & 1 & 0 \\
	0 & 0 & 1
	\end{array} \right) = \left( \begin{array}{ccc}                                   f_{11}^{1} & f_{12}^{2} & f_{13}^{2} \\
	0 & f_{22}^{1} & f_{23}^{1} \\
	0 & 0 & f_{33}^{0}
	\end{array} \right) = O_{3}^{1} 
 \end{displaymath}
 \begin{displaymath}
   \gamma I_{1}^{1} =\left( \begin{array}{ccc}                                   1 & f_{12}^{1} & f_{13}^{2} \\
	0 & 1 & f_{23}^{1} \\
	0 & 0 & 1
	\end{array} \right) \left( \begin{array}{ccc}                                   t^3 & 0 & 0 \\
	0 & t^2 & 0 \\
	0 & 0 & 1
	\end{array} \right) = \left( \begin{array}{ccc}                                   f_{11}^{3} & f_{12}^{3} & f_{13}^{2} \\
	0 & f_{12}^{2} & f_{13}^{1} \\
	0 & 0 & f_{23}^{0}
	\end{array} \right) =
	O_{1}^{1}  
 \end{displaymath}
\begin{displaymath}
	\gamma I_{1}^{2} = \left( \begin{array}{ccc}                                   f_{11}^{0} & f_{12}^{1} & f_{13}^{2} \\
	0 & f_{22}^{0} & f_{23}^{1} \\
	0 & 0 & f_{33}^{0}
	\end{array} \right) \left( \begin{array}{ccc}                                   t^3 & 0 & 0 \\
	0 & t & 0 \\
	0 & 0 & 1
	\end{array} \right) = \left( \begin{array}{ccc}                                   f_{11}^{3} & f_{12}^{2} & f_{13}^{2} \\
	0 & f_{22}^{1} & f_{23}^{1} \\
	0 & 0 & f_{33}^{0}
	\end{array} \right) = I_{1}^{2}
 \end{displaymath}
 \begin{displaymath}
     \gamma I_{2}^{2} =\left( \begin{array}{ccc}                                   f_{11}^{0} & f_{12}^{1} & f_{13}^{2} \\
	0 & f_{22}^{0} & f_{23}^{1} \\
	0 & 0 & f_{33}^{0}
	\end{array} \right) \left( \begin{array}{ccc}                                   t^2 & 0 & 0 \\
	0 & t^2 & 0 \\
	0 & 0 & 1
	\end{array} \right) = \left( \begin{array}{ccc}                                   f_{11}^{2} & f_{12}^{3} & f_{13}^{2} \\
	0 & f_{22}^{2} & f_{23}^{1} \\
	0 & 0 & f_{33}^{0}
	\end{array} \right) = O_{1}^{2} 
 \end{displaymath}
 \begin{displaymath}
   \gamma I_{3}^{2} =\left( \begin{array}{ccc}                                   f_{11}^{0} & f_{12}^{1} & f_{13}^{2} \\
	0 & f_{22}^{0} & f_{23}^{1} \\
	0 & 0 & f_{33}^{0}
	\end{array} \right) \left( \begin{array}{ccc}                                   t & 0 & 0 \\
	0 & 1 & 0 \\
	0 & 0 & 1
	\end{array} \right) = \left( \begin{array}{ccc}                                   f_{11}^{1} & f_{12}^{1} & f_{13}^{2} \\
	0 & f_{22}^{0} & f_{23}^{1} \\
	0 & 0 & f_{33}^{0}
	\end{array} \right) =
	O_{2}^{2}  
 \end{displaymath}

\end{exmpl}
\color{black}

\begin{cor}
	The group $\Gamma$ divides the set of neighbors of vertex
	$[L_{\bar{n}}] \in T$ to $d$ different orbits, such that every
	orbit contains the neighbors of the same degree. The friends of
	the vertex are the fixed points of the action.
\end{cor}

\section{Distance between vertices.}

There are two kinds of distances between two vertices of $B^{0}$.
The first is defined in $B[1]$ and is the distance between the two
as $O$-modules of the same chamber. The second is defined
in $B$ and it is the distance between the two as parts of
chambers.

\begin{defn}
	Let $[L]$, $[L']$ be a vertices of $B$. There is a path of
	vertices $[L] = [L_{1}]$, $[L_{2}]$, $\dots$, $[L_{n}] = [L']$
	such that:
	\begin{enumerate}
		\item The vertices are different in pairs.
		\item For every $1 \le i \le n -1$, the vertices $L_{i}$ and $L_{i+1}$ are in the same chamber (where $L_{i}$, $L_{i+1}$ are representatives of $[L_{i}]$ and $[L_{i+1}]$, respectively).
		\item As $O$-modules, for every $1 \le i \le (n-1)$, $[L_{i+1}: \frac{1}{t}L_{i}] =
		q$.
		\item The path is the minimum path that exists.
	\end{enumerate}
	Then the distance between $[L]$ and $[L']$ will be $n$, and it
	will be denoted by $Dis_{B[1]}(L, L') = n$. In other words, there
	is a path of length $n$ between $[L]$ and $[L']$ with edges of
	color $1$.
\end{defn}
\begin{defn}
	Let $[L]$, $[L']$ be vertices of $B$. Suppose there is a path of
	vertices $[L] = [L_{1}]$, $[L_{2}]$, $\dots$, $[L_{n}] = [L']$
	such that:
	\begin{enumerate}
		\item The vertices are different in pairs.
		\item For every $1 \le i \le n-1$, the vertices $L_{i}$ and
		$L_{i+1}$ are in the same chamber.
		\item For every $i,j$ such that $|i - j| \ne 1$, $L_{i}$ and $L_{j}$ are not in the same chamber.
		\item The path is the minimum path that exists.
	\end{enumerate}
	The number $n$ is the distance between $[L]$ and $[L']$ in $B$ and it will be denoted by $Dis_{B}(L, L') = n$. The sequence
	$[L_{1}]$, $[L_{2}]$, $\dots$, $[L_{n}]$ is called the path in $B$ between $[L]$ and $[L']$.
\end{defn}
\begin{defn}
	Let $[L]$ be a vertex of $B$. The distance of $[L]$ from $T$ is
	the minimal distance in $B$ between $[L]$ to a vertex of $T$. It is
	denoted by $Dis(T, L) = \min_{[L'] \in T} \{Dis_{B}(L, L')\}$ and \\
	$0 \le Dis(T,L) < \infty$.
\end{defn}
\begin{lem}
	\quad{}
	\begin{enumerate}
		\item The distance between vertices $[L_{ \bar{n}}] \in T$ and $[L_{ \bar{m}}] \in T$ in $B$ is $$Dis_{B}(L_{ \bar{n}}, L_{
			\bar{m}}) = \min_{j \in \Z} \max_{1 \le i \le d} \{|n_{i} -m_{i}
		- j| \}. $$\item The distance between vertices $[L_{ \bar{m}}] \in T$ and $[L_{ \bar{n}}] \in T$ in $B[1]$ is
		$$Dis_{B[1]}(L_{ \bar{m}}, L_{ \bar{n}}) = \min_{j \in \Z} \{ \sum_{i = 1}^{d}|n_{i} - m_{i} - j|
		\}.$$\item Let $Dis_{B}(L, L') = n$ be the distance between two vertices $[L]$ and $[L']$. Then there is a path $[L] = [L_{0}], [L_{1}], \dots, [L_{n}] = [L']$ between the two vertices. Denote by $i_{j}$ the color of edge $([L_{j-1}], [L_{j}]) \in B[i_{j}]$, $1 \le j \le n$. Then if $N = \sum_{j=1}^{n}i_{j}$, $Dis_{B[1]}(L, L') = N$.\\
		\item Let $\gamma \in \Gamma$ and $Dis_{B}([L_{0}],[L_{n}])=n$
		with a corresponding path of vertices of $B$,
		$[L_{0}],[L_{1}],\dots, [L_{n}]$. Then $Dis_{B}([\gamma L_{0}],
		[\gamma L_{n}]) = n$ and $[ \gamma L_{0}], [\gamma L_{1}], \dots,
		[\gamma L_{n}]$ is also the corresponding path of vertices of $B$.
	\end{enumerate}
\end{lem}

\begin{prop}
    Define $\Gamma = \left \{ \gamma \in \Gamma _{\bar{n}} : L_{\bar{n}} \in T \textrm{, $\Gamma_{\bar{n}}$ is the stabilizer of $L_{\bar{n}}$} \right \}$. Then
    $B = \bigcup_{\gamma \in \Gamma} \gamma (T)$.
\end{prop}

\begin{proof}
	Let $L$ be a vertex of $B$ with $Dis(T,L)=n$. Then there is a path $[L_{n_{0}}] = [L_{0}], [L_{1}], \dots, [L_{n}] = [L]$, such that $[L_{0}] \in T$ and for every $1 \le i,j \le n$, $i+1 < j$, $[L_{i}] \in B \setminus T$ and $[L_{i}]$, $[L_{j}]$ are not in the same chamber.
	An element $\gamma$ of $\Gamma_{n_{0}}$ acts on $L_{1}$ and transforms it to an element of $T$. Moreover, $\gamma$ preserves on the path, so $\gamma[L_{0}], \gamma[L_{1}], \dots, \gamma[L_{n}]$ is also a path. Then for every $1 \le i \le n$ there is a $\gamma_{i-1}$ such that $\gamma_{i-1} \in \Gamma_{n_{i-1}}$, where $\Gamma_{n_{i-1}}$ is the stabilizer of the vertex $\gamma_{(i-2)} \gamma_{(i-3)} \cdots \gamma_{0}L_{(i-1)}$ and $[\gamma_{i-1} \cdots \gamma_{0}L_{(i)}] \in T$. In particular, $[\gamma_{(n-1)}\cdots \gamma_{0}L_{n}] \in T$. Therefore, $\dom{\Gamma}{B} \subseteq T$.
\end{proof}

\section[T is the Fundamental domain]{Disjointness of the distinguished subset $T$ of $B$ from its images under  $\Gamma$}

\begin{lem}
	If two elements of $T$, $\Lambda_{ \bar{n}}$ and $\Lambda_{ \bar{m}}$,
	satisfy $\gamma L_{ \bar{n}} = L_{ \bar{m}}o$ for some $o \in
	\PGL_{d}(\mathcal {O})$, $\gamma \in \Gamma$ and $L_{ \bar{n}} \in \Lambda_{ \bar{n}}$, $L_{
		\bar{m}} \in \Lambda_{ \bar{m}}$, then $\bar{n} = \bar{m}$.
\end{lem}

\begin{proof}
	Let $L_{\bar{n}}$ and $L_{\bar{m}}$ be lattices of $\Lambda_{
		\bar{n}}$ and $\Lambda_{\bar{m}}$ in $T$, respectively.
	Suppose $\gamma L_{ \bar{n}} = L_{ \bar{m}}o$ for $\gamma \in \Gamma$
	and $ o \in \PGL_{d}(O)$. Then $(L_{\bar{m}})^{-1}\gamma L_{\bar{n}} = o$
	or $(t^{-m_{i}})(\gamma_{ij})(t^{n_{j}}) = o_{ij}$. Then
	$(t^{n_{j}-m_{i}} \gamma_{ij})=o_{ij}$, which means that
	$v(\gamma_{ij}) - n_{j} + m_{i} \ge 0$. Then
	$v(\gamma_{ij}) \ge n_{j} - m_{i}$, but $v(\gamma_{ij}) \le 0$ so
	$n_{j} \le m_{i}$.\\
	On the other hand, if $\gamma L_{\bar{n}} = L_{ \bar{m}}o$ then
	$\det( \gamma)t^{ \sum n_{i}} = t^{ \sum m_{j}}det(o)$ and then,
	by valuation,
	\begin{displaymath}
		v(\det( \gamma)t^{ \sum n_{i}}) = v(t^{ \sum
			m_{j}}det(o)),
	\end{displaymath}
	\begin{displaymath}
		0 - \sum_{i=1}^{d}n_{i} = - \sum_{i=1}^{d}m_{i} + 0,
	\end{displaymath}
	\begin{displaymath}
		\sum_{i=1}^{d}m_{i} = \sum_{i=1}^{d}n_{i}.
	\end{displaymath}
	So, for every $1
	\le i \le d$, they satisfy $n_{i} = m_{i}$.
\end{proof}

\begin{prop}
	The set $T$ is the fundamental domain of $B$ under the action of $\Gamma$.
\end{prop}

\begin{proof}
The first condition was proved by using the action
	of a stabilizer of a vertex in $T$ on the neighbors (which are not all in $T$) by Prop. 4.5.
The second condition, was proved by that if $\gamma
	\Lambda_{ \bar{n}} = \Lambda_{ \bar{m}}$ in $T$, then $\bar{n} = \bar{m}$ by Lemma 5.1.    
\end{proof}
	 
\begin{prop}
	The lattice $\Gamma$ is a maximal discrete lattice of $G$.
\end{prop}

\begin{proof}
	
		Assume there is a lattice $\tilde{\Gamma} \ne G$, such that $\Gamma \subseteq \tilde{ \Gamma} \subset G$ and $\tilde{ \Gamma}$
		is a discrete lattice. The lattice is of the form $\sg{\Gamma,
			\delta}$, where $\delta \in G \setminus \Gamma $. The aim is to
		find a sequence $\tilde{ \gamma_{n}} \in \tilde{\Gamma}$ such
		that $lim_{n \rightarrow \infty}\{Dis(\gamma_{n},1)\}=I$. The
		group $G$ can be written as the computation $\Gamma T K$ (because
		$T$ is the fundamental domain of $B$ under the action of $\Gamma$
		and $B = G/K$). Then, for every $g \in G$ there exist $\gamma \in
		\Gamma$, $ L_{ \bar{n}}$, $[L_{\bar{n}}] \in T$ and $k \in F$ such
		that $g = \gamma L_{\bar{n}} k$. Moreover, $F$ can be written as
		$F = I + \frac{1}{t}O$. The requirement is to find $\tilde{
			\gamma} \in \tilde{\Gamma}$ such that $\tilde{ \gamma_{n}} - I
		\equiv 0 \pmod{t^{-n}}$. For example, if $d = 2$ and $k = I$, then
		$\delta = \left ( \begin{array}{cc}
			t^{n_{1}} & 0 \\
			0 & 1
		\end{array} \right
		)$,
		take $\delta ^{-1} = \left ( \begin{array}{cc}
			\frac{1}{t^{n_{1}}} & 0 \\
			0 & 1
		\end{array} \right
		)$.
		Then $\delta \gamma \delta^{-1} = \left ( \begin{array}{cc}
			1 & t \gamma_{12} \\
			t^{-n_{1}} \gamma_{21} & 1
		\end{array} \right )$ hence choose $\gamma = \left ( \begin{array}{cc}
			1 & 0 \\
			1 & 1
		\end{array} \right )$.
		Multiply the element above by itself $n$ times and get $(\delta
		\gamma \delta^{-1})^{n} = \left ( \begin{array}{cc}
			1 & 0 \\
			(t^{-n_{1}})^{n} & 1
		\end{array} \right ) \longrightarrow_{n \to \infty} I $.
	
\end{proof}

\section{The volume of the quotient of a Bruhat-Tits building by a non-uniform lattice}

The volume of a diagram is computed by summing the weights of all
the vertices, where the weight of a vertex is depeneds on its stabilizer  
\begin{displaymath}
	w(\Lambda_{ \bar{n}}) = \frac{1}{|\Gamma_{ \bar{n}}|}
\end{displaymath}

Of course, if the diagram is finite and the weights
are finite, the volume will be finite. But if the diagram is
infinite, then the volume can be infinite too. $\Gamma$ is a
non-uniform lattice, which means that $T$, as the fundamental
domain over $\Gamma$, is infinite. We compute below the total
weight of $T$, which is the co-volume of $\Gamma$.

Define the weight function of vertex $\Lambda_{ \bar{n}}$ ($w:B^{0}\to (0,1]$) by:
\begin{displaymath}
	w(\Lambda_{ \bar{n}}) = \frac{1}{|\Gamma_{ \bar{n}}|}
\end{displaymath}
(where we recall that $\Gamma_{\bar{n}}$ is the stabilizer of the
vertex) and denote by $\bar{n}_{d}$ a series suitable to the
partition $d$. The sum of all weights of vertices of the graph $T$
is: \begin{eqnarray*} w(T) & = & \sum_{ \bar{n}}
	\frac{1}{|\Gamma_{ \bar{n}}|}
\end{eqnarray*}
\begin{eqnarray*}
	& = & \sum_{d} \left[ \sum_{\bar{n}_{d}} \left(
	\frac{\prod_{i=1}^{r}|\GL_{d_{i}}(\mathbb{F}_{q})|q^{ \sum_{i<j
				\le r}d_{i}d_{j} \left( \sum_{l=i}^{j-1}m_{l}+1 \right) }}{q-1}
	\right)^{-1} \right] \\
	& = & (q-1) \sum_{d} \left[ \frac{1}{
		\prod_{i=1}^{r}|\GL_{d_{i}}(\mathbb{F}_{q})|} \left( \sum_{
		\bar{n}_{d}}q^{- \sum_{i<j \le r}d_{i}d_{j} \left(
		\sum_{l=i}^{j-1}m_{l} + 1 \right)} \right) \right] 
	\\
	& = &
	(q-1) \sum_{d} \Bigg[ \frac{1}{ \prod_{i=1}^{r}|\GL_{d_{i}}(\mathbb{F}_{q})|}
	\Bigg( \sum_{ \begin{array}{ccc} g_{1}, g_2, \dots,g_r = 1                   \end{array}}^{ \infty} q^{- \sum_{i<j \le r}d_{i}d_{j} \left( \sum_{l=i}^{j-1}m_{l} + 1 \right)} \Bigg) \Bigg]
	\\ & = &
	(q-1) \sum_{d} \Bigg[ \frac{1}{ \prod_{i=1}^{r}|\GL_{d_{i}}(\mathbb{F}_{q})|} \cdot \\
	& & \Bigg( \sum_{ \begin{array}{ccc} g_{1},g_2,\dots,g_r = 1
	\end{array}}^{ \infty} \Big[ q^{- \sum_{i<j \le r}d_{i}d_{j}} \prod_{l=1}^{r-1}q^{- \left( \sum_{i=1}^{l} \sum_{j=l+1}^{r}d_{i}d_{j} \right) m_{l}} \Big] \Bigg) \Bigg],
\end{eqnarray*}
which is equal to
\begin{eqnarray*} &  &
	(q-1) \sum_{d} \left[ \frac{1}{
		\prod_{i=1}^{r}|\GL_{d_{i}}(\mathbb{F}_{q})|}q^{- \sum_{i<j \le
			r}d_{i}d_{j}} \prod_{l=1}^{r-1} \left( \sum_{  m_{l}= 1}^{ \infty}
	q^{- \left( \sum_{i=1}^{l} \sum_{j=l+1}^{r}d_{i}d_{j} \right)
		m_{l}} \right) \right] \\ & = & (q-1) \sum_{d} \left( \frac{1}{
		\prod_{i=1}^{r}|\GL_{d_{i}}(\mathbb{F}_{q})|}q^{- \sum_{i<j \le
			r}d_{i}d_{j}} \prod_{l=1}^{r-1} \frac{-q^{- \left( \sum_{i=1}^{l}
			\sum_{j=l+1}^{r}d_{i}d_{j} \right)}}{q^{- \left( \sum_{i=1}^{l}
			\sum_{j=l+1}^{r}d_{i}d_{j} \right)}-1} \right)
\end{eqnarray*}
where the sum over $d$ is a sum over different partition of $d$, $d =
\sum_{i=1}^{r}d_{i}$.

For every given group $G$ the size $d$ is fixed, so the sum and
the product over partition of $d$ are finite. Moreover, the elements of
the product in the right side are positive numbers smaller than $1$ and the
elements of the sum in left are finite. We get that $w(T) < \infty$.

\section{Explicit computations for quotients of type $\tilde{A_{d}}$}

For
building of type $\tilde{A}_d$ we can define Hecke operators
$A_{1}, \cdots, A_{d-1}$ by the $d-1$ colors of the their
vertices. These operators commute, so when we consider
representations of the Hecke algebra we are in fact looking for
simultaneous eigenvectors, which determine the simultaneous
spectrum, denoted by $S_d$.
A finite quotient $X$ of $B$ is called Ramanujan complex if the eigenvalues of every non-trivial simultanenous eigenvector $v \in L^2 (X)$, $A_{k}v=\lambda_{k}v$, satisfy $(\lambda_1, \lambda_2, \dots, \lambda_{d-1})\in S_d$.  
 
  A Ramanujan graph, $X$, is defined as a finite $k$-regular graph,
which is connected and its adjacency matrix, $A_{X}$ satisfies: \\
\begin{displaymath}
	\spec (\mathcal{A}_{X}) \subseteq \{ \pm k \}\bigcup[-2 \sqrt{(k-1)},2
	\sqrt{(k-1)}].
\end{displaymath}
Samuels in \cite{Sam} proved that a quotient of the building by a
non-uniform lattice is not Ramanujan. We will use Hecke operators
to exhibit in details the quotient in the case $d=3$.

In the $1$-dimensional case of a graph, the Hecke operator is
defined for functions on the set of vertices by summing for each
vertex over its neighbors.
For high dimensional complexes, one
may apply the extra symmetry to define operators which take into
account more than the skeleton graph structure.

The weight function, $w: B^{1} \bigcup B^{0} \to (0,1]$ is defined by
if $u,v \in B^{0}$ and $(u,v) \in B[i]$ then:
$w(u)= | \Gamma(u)|^{-1}$ and $w(u,v) = |\Gamma(u)\bigcap
\Gamma(v)|^{-1}$. Define a measure on $B$ by the weight function $\mu(S) = \sum_{v
	\in S}w(v)$ for every $S \subseteq B^{0}$. This gives rise to the
space $L^2(B)$ of square-integrable functions with respect to the
measure, in which the norm and inner product are defined by:
$$\langle f, g \rangle = \int_{B^{0}}f \bar{g}d \mu = \sum_{v \in
	B^{0}}w(v)f(v) \bar{g}(v).$$ The Hecke operators are defined for
$f \in L^2(B)$ by $$(A_{i}f)(u) = \sum_{(u,v) \in B[i]}f(v).$$ It
can be verified that the Hecke operators $A_{i}$ and $A_{d-i}$ are
conjugates. Moreover, the operators are bounded, normal and
commute and their simultaneous spectrum, $S$, is a subset of $\mathbb{C}^{d-1}$.

Let $\Gamma$ be a non-uniform lattice as above, the Hecke
operators defined over $L^{2}(B)$ induce operators on
$L^{2}(\dom{\Gamma}{B})$.
The Hecke operators acting on $L^{2}(
\dom{\Gamma}{B})$ are
$(A_{i}f)(u) = \sum_{(u,v) \in B[i]} \frac{w(u,v)}{w(u)}f(v)$.
\begin{rem}
	The Hecke operators $A_i$ and $A_{d-i}$ on $L^{2}(
	\dom{\Gamma}{B})$ are conjugate (because $(u,v) \in B[i]$
	iff $(v,u) \in B[d-i]$). The operators $A_i,A_j$ are also known to
	commute.
\end{rem}

\subsection{The case $d=2$}

\begin{exmpl} Let $d=2$.
	In this example we find the eigenvector
	f, which can be in $L^{2}( \Gamma \setminus B)$ of the Hecke operator $A_{1}$,
	which is the only one.
	
	We note that a complex number $\lambda$ is in the spectrum of an operator
	$T$ iff $T-\lambda I$ is not onto. Here we compute eigenvalues of the operator.
	The order of the
	stabilizers of the vertices of $T$ are $$w(\Lambda_{ \bar{0}}) =
	\frac{1}{(q-1)^{2}q(q+1)},  w( \Lambda_{ \bar{n}}) =
	\frac{1}{(q-1)^{2}q^{n+1}}.$$ The eigenvector $f$ in this case has
	to satisfy $$\frac{1}{(q-1)^{2}q(q+1)}|f_{0}|^{2}+ \sum_{n=1}^{ \infty}
	\frac{1}{(q-1)^{2}q^{n+1}}|f_{n}|^{2} < \infty.$$ The operator
	$A_{1}$ is acting on the neighbors of color $1$. Let $f = (f_{0},
	f_{1}, \dots)$ be an element of $L^{2}( \dom{\Gamma}{B})$ then
	$(A_{1}f)_{0}=f_{1}(q+1)$ and $(A_{1}f)_{n} = qf_{n-1}+f_{n+1}$.
	
	We are looking for a vector $f=(f_{0}, f_{1}, \dots)$ of $\lambda$ such that
	$(A_{1}f)= \lambda f$. We know that
	\begin{displaymath}
		(\lambda f)_{0} = (q+1)f_{1}, (\lambda f)_{n} = q
		f_{n-1}+f_{n+1}.
	\end{displaymath} Hence, $f_{n+1} = \lambda f_{n} - qf_{n-1}$ and
	know $f_{0} = C + D = 1$ or $D = 1 - C$. This relation can be
	solved as $f_{n} = Cr_{1}^{n} + Dr_{2}^{n}$, where $r_{1}$ and
	$r_{2}$ are the roots of the equation $r^{2} - \lambda r + q = 0$, therefore,
	$r_{1,2} = \frac{ \lambda \pm \sqrt { \lambda ^{2} -4q}}{2}$.
	
	From the equation $\lambda (C + D) = (q + 1)f_{1} = (q + 1)(Cr_{1} + Dr_{2})$
	we get that
	\begin{displaymath}
		\lambda = (q + 1)(Cr_{1}+(1 - C)r_{2}) = (q + 1)(r_{2} +
		(r_{1} - r_{2})C) =
	\end{displaymath}
	\begin{displaymath}
		(q + 1)r_{2} + (q + 1) \sqrt {\lambda^{2} -
			4q}C,
	\end{displaymath} and then $C = \frac{\lambda - (q+1)r_{2}}{(q+1)\sqrt
		{\lambda^{2} - 4q}}$.
	
	The eignvector $f$ is then the sequence $(f_{0}, f_{1}, \dots)$,
	where  $$f_{n} = ( \frac{\lambda - (q+1)r_{2}}{(q+1)\sqrt
		{\lambda^{2} - 4q}})( \frac{ \lambda + \sqrt
		{\lambda^{2}-4q}}{2})^{n}+(1 - \frac{\lambda -
		(q+1)r_{2}}{(q+1)\sqrt {\lambda^{2} - 4q}})( \frac{ \lambda -
		\sqrt {\lambda^{2}-4q}}{2})^{n}.$$
	
	The function $f = (f_{0}, f_{1}, \dots)$ has to satisfy
	\begin{displaymath}
		\frac{1}{(q-1)^{2}q(q+1)}||f_{0}||^{2}+ \sum_{n=1}^{\infty}
		\frac{1}{(q-1)^{2}q^{n+1}}||f_{n}||^{2} < \infty.
	\end{displaymath}
	Now, if
	$r_{1},r_{2} \in \mathbb{C} \backslash \mathbb{R}$ then
    \begin{eqnarray*}
        ||f_{n}||^{2} = \\
        & = & f_{n} \bar{f_{n}} \\
        & = & (Cr_{1}^{n}+Dr_{2}^{n})(
		\bar{C}r_{2}^{n}+ \bar{D}r_{1}^{n}) \\
        & = &
		C \bar{D}r_{1}^{2n} + C\bar{C}(r_{1}r_{2})^{n} + D \bar{C}r_{2}^{2n}+D
		\bar{D}(r_{2}r_{1})^{n} \\
        & = &
		q^{n}(C \bar{C} + D \bar{D}) + C
		\bar{D}( \frac{ \lambda ^{2} + \lambda \sqrt
			{\lambda^{2}-4q}-2q}{2})^{n}+D \bar{C}( \frac{ \lambda ^{2} -
			\lambda \sqrt {\lambda^{2}-4q}-2q}{2})^{n}.
    \end{eqnarray*}
	If $r_{1}$, $r_{2} \in \mathbb{R}$ then \\
	\begin{eqnarray*}
		||f_{n}||^{2}
		& = & f_{n}f_{n} = \\
        & = & C^{2}( \frac{ \lambda ^{2} + \lambda \sqrt
			{\lambda^{2}-4q}-2q}{2})^{n}
		+ D^{2}( \frac{ \lambda ^{2} -
			\lambda \sqrt {\lambda^{2}-4q}-2q}{2})^{n}
		+ q^{n}2CD.
	\end{eqnarray*}
	The degree
	of the elements $||f_{n}||^{2}$ is $n$, and
	$\deg((q-1)^{2}q^{n+1}) = n+3$. Then
	$\deg\frac{||f_{n}||^{2}}{(q-1)^{2}q^{n+1}} = -3$. It is unclear if $f$  belongs
	$L^{2}( \dom{\Gamma}{B})$.
\end{exmpl}
\subsection{The case $d=3$}\label{ex:wgt}
	Assume $d = 3$. The diagram of $T$ in that case is like an infinite
	triangle with a vertex in $L_{ \bar{0}}$. The meaning of an arrow
	from $[L_{\bar{n}^{1}}]$ to $[L_{\bar{n}^{2}}]$ is that
	$[L_{\bar{n}^{1}}]$ is a neighbor of  $[L_{\bar{n}^{2}}]$ of
	degree $2$ and $[L_{\bar{n}^{2}}]$ is a neighbor of
	$[L_{\bar{n}^{1}}]$ of degree $1$. Denote the arrows and the
	vertices of the diagram by colors. There are $4$ kinds of vertices and
	$12$ kinds of arrows (which are divided by the construction of the
	stabilizers). 
	
	
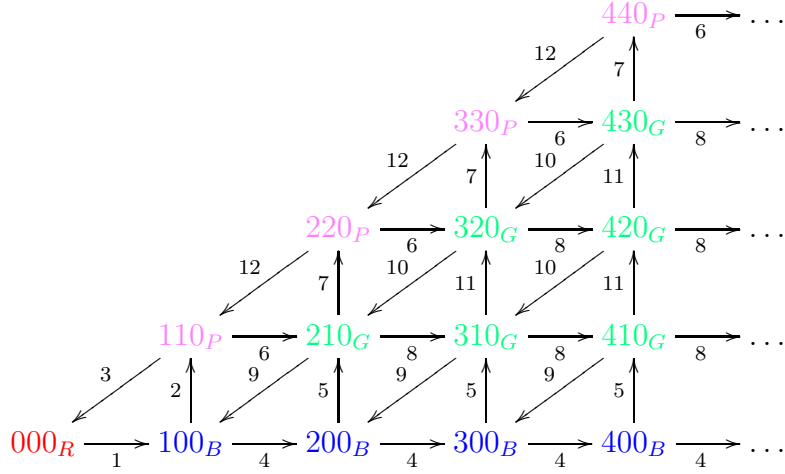
\begin{displaymath}
		\xymatrix{
			& & & & \textcolor{Pink}{440_P} \ar[dl]_{12} \ar[r]_{6} & \dots\\
			& & & \textcolor{Pink}{330_P} \ar[dl]_{12} \ar[r]_{6} & \textcolor{Green}{430_G} \ar[dl]_{10} \ar[u]^{7} \ar[r]_{8} & \dots \\
			& & \textcolor{Pink}{220_P} \ar[dl]_{12} \ar[r]_{6} & \textcolor{Green}{320_G} \ar[dl]_{10} \ar[u]^{7} \ar[r]_{8} & \textcolor{Green}{420_G} \ar[dl]_{10} \ar[u]^{11} \ar[r]_{8} & \dots \\
			& \textcolor{Pink}{110_P} \ar[dl]_{3} \ar[r]_{6} & \textcolor{Green}{210_G} \ar[dl]_{9} \ar[u]^{7} \ar[r]_{8} & \textcolor{Green}{310_G} \ar[dl]_{9} \ar[u]^{11} \ar[r]_{8} & \textcolor{Green}{410_G} \ar[dl]_{9} \ar[u]^{11} \ar[r]_{8} & \dots \\
			\colored{000_R} \ar[r]_{1} & \textcolor{Blue}{100_B} \ar[u]^{2} \ar[r]_{4} & \textcolor{Blue}{200_B} \ar[u]^{5} \ar[r]_{4} & \textcolor{Blue}{300_B} \ar[u]^{5} \ar[r]_{4} & \textcolor{Blue}{400_B} \ar[u]^{5} \ar[r]_{4} & \dots }
	\end{displaymath}
	\\
	\begin{displaymath}
		\bar{n} = \left\{ \begin{array}{ccc}
			\bar{0} & \colored{Red}\\
			n_{1}n_{1}0 & \textcolor{Pink}{Pink}\\
			n_{1}00 & \textcolor{Blue}{Blue}\\
			n_{1}n_{2}0 & \textcolor{Green}{Green},&n_{1}>n_{2}
		\end{array}\right.
	\end{displaymath}
	
	The stabilizers of the $4$ kinds of vertices ($\bar{n}=n_{1}n_{2}0$) are:\\
	$| \Gamma_{\colored{000}}| = (q-1)^{2}(q^{2}+q+1)(q+1)q^{3}$,\\
	$| \Gamma_{\textcolor{Pink}{(n_1n_10)}}| = (q-1)^{2}(q+1)q^{2n_{1} + 3}$,\\
	$| \Gamma_{\textcolor{Blue}{(n_100)}}| = (q-1)^{2}(q+1)q^{2n_{1} + 3}$,\\
	$| \Gamma_{\textcolor{Green}{(n_1n_20)}}| = (q-1)^{2}q^{2n_{1} + 3}$.\\
	Compute the indexes of the operators $A_{1}$ and $A_{2}$ by the
	weight function, $\frac{w(u,v)}{w(u)} = \frac{
		|\Gamma(\bar{n}^{1})|}{|\Gamma(\bar{n}^{1}, \bar{n}^{2})|}$, where
	$u = [L_{\bar{n}^{1}}]$ and $v = [L_{\bar{n}^{2}}]$ and
	$\Gamma(\bar{n})$ is the stabilizer of $[L_{\bar{n}}]$ and
	$\Gamma(\bar{n}^{1},\bar{n}^{2})$ is the stabilizer of the edge
	between $[L_{\bar{n}^{1}}]$ and $[L_{\bar{n}^{2}}]$.\\
	\\
	\begin{tabular}{|r|c|c|c|c|l|}
		\hline
		Edge & $n_{1}^{1}n_{2}^{1}0$ & $n_{1}^{2}n_{2}^{2}0$ & $w(u,v)$ & $\frac{w(u,v)}{w(u)}$ & $\frac{w(u,v)}{w(v)}$\\
		\hline
		$1$ & $000$ & $100$ & $(q+1)(q-1)^{2}q^{3}$ & $q^{2} + q + 1$ & $q^{2}$\\
		\hline
		$2$ & $100$ & $110$ & $(q-1)^{2}q^{4}$ & $(q+1)q$ & $ (q+1)q$\\
		\hline
		$3$ & $110$ & $000$ & $(q+1)(q-1)^{2}q^{3}$ & $q^{2}$ &  $q^{2} + q + 1$\\
		\hline
		$4$ & $n_{1}00$, $n_{1} \ge 1$ & $(n_{1}+1)00$ & $ (q+1)(q-1)^{2}q^{2n_{1}+3}$ & 1 & $q^{2}$\\
		\hline
		$5$ & $n_{1}00$ & $n_{1}10$ & $ (q-1)^{2}q^{2n_{1}+2}$ & $(q + 1)q$ & q\\
		\hline
		$6$ & $n_{1}n_{1}0$ & $(n_{1}+1)n_{1}0$ & $(q-1)^{2}q^{2n_{1}+3}$ & $q+1$ & $q^{2}$\\
		\hline
		$7$ & $n_{1}(n_{1} - 1)0$ & $n_{1}n_{1}0$ & $ (q-1)^{2}q^{2n_{1}+2}$ & $q$ & $(q+1)q$\\
		\hline
		$8$ & $n_{1}n_{2}0$ & $(n_{1}+1)n_{2}0$ & $ (q-1)^{2}q^{2n_{1}+3}$ & 1 & $q^{2}$\\
		\hline
		$9$ & $n_{1}10$ & $(n_{1} -1)00$ & $ (q-1)^{2}q^{2n_{1} + 1}$ & $q^{2}$ & $q + 1$\\
		\hline
		$10$ & $n_{1}n_{2}0$ & $(n_{1}-1)(n_{2}-1)0$ & $(q-1)^{2}q^{2n_{1}+1}$ & $q^{2}$ & $1$\\
		\hline
		$11$ & $n_{1}n_{2}0$ & $n_{1}(n_{2}+1)0$ & $(q-1)^{2}q^{2n_{1}+2}$ & $q$ & $q$\\
		\hline
		$12$ & $n_{1}n_{1}0$ & $(n_{1} - 1)(n_{1} - 1)0$ & $ (q+1)(q-1)^{2}q^{2n_{1}+1}$ & $q^{2}$ & 1\\
		\hline
	\end{tabular}
	\\
	\\
	Since $d=3$, there are two colored Hecke operators,
	$A_{1}$ and $A_{2}$. We will check that they commute, namely that
	$(A_{1}A_{2}f)(u) = (A_{2}A_{1}f)(u)$ for every $u \in T$. First
	compute the left-hand side, with the understanding that all the points $u,r,v$ are in $T$:
	\begin{eqnarray*}
		(A_{2}A_{1}f)(u) =& \sum_{(u,r) \in B[2]} \frac{w(u,r)}{w(u)}(A_{1}f)(r) \\ 
	=&\sum_{(u,r) \in B[2]} \frac{w(u,r)}{w(r)}\left (\sum_{(r,v) \in B[1]} \frac{w(r,v)}{w(r)} \right )f(v)\\ 
	=&\sum_{ 	(u,r), (v,r) \in B[2]} \frac{w(u,r)}{w(u)} \frac{w(r,v)}{w(r)}f(v) \\
	=&\sum_{(u,r), (v,r) \in B[2]} \frac{| \Gamma(u)|}{| \Gamma(u,r)|} \frac{| \Gamma(r)|}{| \Gamma(r,v)|}f(v),
	\end{eqnarray*}
	\begin{eqnarray*}
		(A_{1}A_{2}f)(u) = & \sum_{ (u,t), (v,t) \in B[1]} \frac{| \Gamma(u)|}{| \Gamma(u,t)|} \frac{| \Gamma(t)|}{| \Gamma(t,r)|}f(v).\\
	\end{eqnarray*}
	\\
	Dividing through by $| \Gamma(u)|$, we have to show that
	\begin{equation}\label{eq:wgt}
		\sum_{ (u,t),(v,t) \in B[1]} \frac{|
			\Gamma(t)|}{| \Gamma(u,t)|| \Gamma(v,t)|}f(v)\\	
		= \sum_{(u,r),(v,r) \in B[2]} 
\frac{| \Gamma(r)|}{| \Gamma(u,r)||
			\Gamma(v,r)|}f(v).\\
	\end{equation}
	The proof will be done with one case of internal vertex (namely,
	the vertex does not lie on the side walls of $T$),
	the other cases can be proved in the same way. Let $u$ be
	a vertex of $T$ of color green, such that $u
	= L_{n_{1}n_{2}0}$ and $n_{1} \ge 6$, $n_{2} \ge 3$ and $n_{1} -
	n_{2} \ge 3$.
	
	To save space, denote here the number $(q-1)^{2}q^{m}$ by $(m)$.
	Then, by knowing the coordinates of vertex $u$, the computation
	at the left-hand side of equation~(\ref{eq:wgt}) is:
	\begin{eqnarray*}
		\sum_{ \begin{array}{l}
				(u,t),(v,t) \in B[1]\\
				u,t,v \in T \end{array}} \frac{| \Gamma(t)|}{| \Gamma(u,t)||
			\Gamma(v,t)|}f(v) &  = & \\
		=   \frac{|
			\Gamma(t_{1})|}{|\Gamma(u,t_{1})||\Gamma(v_{1},t_{1})|}f(v_{1})
		& + & \frac{|
			\Gamma(t_{1})|}{|\Gamma(u,t_{1})||\Gamma(v_{6},t_{1})|}f(v_{6})
		\\
		+ \frac{|
			\Gamma(t_{2})|}{|\Gamma(u,t_{2})||\Gamma(v_{2},t_{2})|}f(v_{2})
		& + &
		\frac{|
			\Gamma(t_{2})|}{|\Gamma(u,t_{2})||\Gamma(v_{3},t_{2})|}f(v_{3})
		\\
		+ \frac{| \Gamma(t_{3})|}{|\Gamma(u,t_{3})||\Gamma(v_{4},t_{3})|}f(v_{4})
		& + &
		\frac{| \Gamma(t_{3})|}{|\Gamma(u,t_{3})||\Gamma(v_{5},t_{3})|}f(v_{5})
		\\
		= \frac{(2n_{1}+5)}{(2n_{1}+3)(2n_{1}+5)}f(v_{1})
		& + &
		\frac{(2n_{1}+5)}{(2n_{1}+3)(2n_{1}+4)}f(v_{6})
		\\
		+ \frac{(2n_{1}+3)}{(2n_{1}+2)(2n_{1}+3)}f(v_{2})
		& + & \frac{(2n_{1}+3)}{(2n_{1}+2)(2n_{1}+1)}f(v_{3})
		\\
		+ \frac{(2n_{1}+1)}{(2n_{1}+1)(2n_{1}-1)}f(v_{4})
		& + & \frac{(2n_{1}+1)}{(2n_{1}+1)(2n_{1})}f(v_{5})
		\\
		= \frac{1}{(2n_{1}+3)}f(v_{1})
		& + & \frac{1}{(2n_{1}+2)}f(v_{6})
		\\
		+ \frac{1}{(2n_{1}+2)}f(v_{3})
		& + & \frac{1}{(2n_{1})}f(v_{2})
		\\
		+ \frac{1}{(2n_{1}-1)}f(v_{4})
		& + & \frac{1}{(2n_{1})}f(v_{5}).
	\end{eqnarray*}
 by out conditions, all the star below is in $T$ and of color
green.
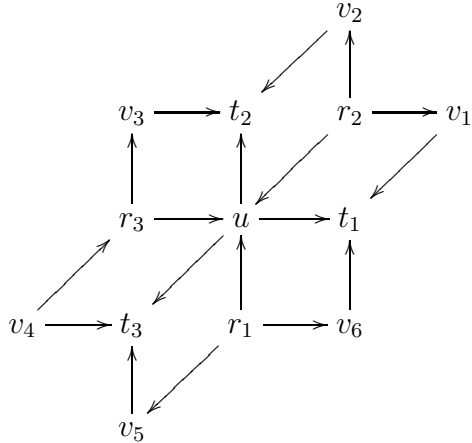
\begin{displaymath}
\xymatrix{
	& & & v_{2} \ar[dl] & \\
	& v_{3} \ar[r] & t_{2} & r_{2} \ar[dl] \ar[r] \ar[u] & v_{1} \ar[dl] \\
	& r_{3} \ar[r] \ar[u] & u \ar[r] \ar[u] \ar[dl] & t_{1} & \\
	v_{4} \ar[r] \ar[ur] & t_{3} & r_{1} \ar[u] \ar[r] \ar[dl] & v_{6} \ar[u] & \\
	& v_{5} \ar[u] & & }
\end{displaymath}

	Now, compute the right-hand side of equation~(\ref{eq:wgt}) and get:
	\begin{eqnarray*}  \sum_{ \begin{array}{l}
				(u,r),(v,r) \in B[2]\\
				u,r,v \in T \end{array}} \frac{| \Gamma(r)|}{| \Gamma(u,r)|| \Gamma(v,r)|}f(v)& =&\\
	\end{eqnarray*}
	\begin{eqnarray*}
		= \frac{| \Gamma(r_{1})|}{|\Gamma(u,r_{1})||\Gamma(v_{5},r_{1})|}f(v_{5}) & + &
		\frac{| \Gamma(r_{1})|}{|\Gamma(u,r_{1})||\Gamma(v_{6},r_{1})|}f(v_{6}) \\
		+ \frac{| \Gamma(r_{2})|}{|\Gamma(u,r_{2})||\Gamma(v_{1},r_{2})|}f(v_{1}) & + &
		\frac{| \Gamma(r_{2})|}{|\Gamma(u,r_{2})||\Gamma(v_{2},r_{2})|}f(v_{2})  \\
		+ \frac{| \Gamma(r_{3})|}{|\Gamma(u,r_{3})||\Gamma(v_{3},r_{3})|}f(v_{3}) & + &
		\frac{| \Gamma(r_{3})|}{|\Gamma(u,r_{3})||\Gamma(v_{4},r_{3})|}f(v_{4})  \\
		= \frac{(2n_{1}+3)}{(2n_{1}+1)(2n_{1}+2)}f(v_{5}) & + &
		\frac{(2n_{1}+3)}{(2n_{1}+2)(2n_{1}+3)}f(v_{6}) \\
		+ \frac{(2n_{1}+5)}{(2n_{1}+3)(2n_{1}+5)}f(v_{1}) & + &
		\frac{(2n_{1}+5)}{(2n_{1}+4)(2n_{1}+3)}f(v_{2}) \\
		+ \frac{(2n_{1}+1)}{(2n_{1})(2n_{1}+1)}f(v_{3}) & + &
		\frac{(2n_{1}+1)}{(2n_{1}-1)}(2n_{1}+1)f(v_{4}) \\
		= \frac{1}{(2n_{1})}f(v_{5}) & + & \frac{1}{(2n_{1}+2)}f(v_{6}) \\
		+ \frac{1}{(2n_{1}+3)}f(v_{1})  & + & \frac{1}{(2n_{1}+2)}f(v_{2}) \\
		+ \frac{1}{(2n_{1})}f(v_{3}) & + & \frac{1}{(2n_{1}-1)}f(v_{4}).
	\end{eqnarray*}
	
	We get that the left side and the right side of equation~(\ref{eq:wgt}) are equal.

The spectrum of the Hecke operators is important for understanding
expansion properties of the graph.

\begin{defn}
	Denote by $\spec(\dom{\Gamma}{B})$ the set of simultaneous
	non-trivial eigenvalues of the $d-1$ Hecke operators on
	$L^{2}((\dom{\Gamma}{B})^{0}$) and by $\spec(B)$ the spectrum of
	the $d-1$ Hecke operators on $L^{2}(B^{0})$.
\end{defn}
The complex $\dom{\Gamma}{B}$ is Ramanujan if its spectrum
satisfies:
\begin{displaymath}
	\spec(\dom{\Gamma}{B}) \subseteq \spec(B)
\end{displaymath}

\begin{exmpl}
	Assume $T$ satisfies the condition of Example~\ref{ex:wgt}, so in particular $d=3$. We will compute the Hecke operators and the first few values of a simultaneous eigenvector $f \in L^{2}(\Gamma \setminus B)$ of the operators.
	\\
	Denote the number $q^{2}+q+1$ by $t$ and the number $q+1$ by $r$. The eigenvalues $\lambda_{1}$ and $\lambda_{2}$ belong to $\mathbb{C}$.\\
	The operators of Hecke, $A_{1}$ and $A_{2}$, which act on the function $f$, are:\\
	\\
	\\
	\\
	$(A_{1}f)(000) = \lambda_{1}f(000) = tf(100)$\\
	$(A_{2}f)(000) = \lambda_{2}f(000) = tf(110)$\\
	\\
	$(A_{1}f)(100) = \lambda_{1}f(100) = rqf(110) + f(200)$\\
	$(A_{2}f)(100) = \lambda_{2}f(100) = rf(210) + q^{2}f(000)$\\
	\\
	$(A_{1}f)(110) = \lambda_{1}f(110) = rf(210) + q^{2}f(000)$\\
	$(A_{2}f)(110) = \lambda_{2}f(110) = f(220) + rqf(100)$\\
	\\
	$(A_{1}f)(200) = \lambda_{1}f(200) = f(300) + qrf(210)$\\
	$(A_{2}f)(200) = \lambda_{2}f(200) = rf(310) + q^{2}f(100)$\\
	\\
	$(A_{1}f)(210) = \lambda_{1}f(210) = f(310) + qf(220) + q^{2}f(100)$\\
	$(A_{2}f)(210) = \lambda_{2}f(210) = f(320) + qf(200) + q^{2}f(110)$\\
	\\
	$(A_{1}f)(220) = \lambda_{1}f(220) = rf(330) + q^{2}f(110)$\\
	$(A_{2}f)(220) = \lambda_{2}f(220) = f(330) + qrf(210)$\\
	\\
	$(A_{1}f)(n_{1}00) = \lambda_{1}f(n_{1}00) = f((n_{1}+1)00) + rqf(n_{1}10)$\\
	$(A_{2}f)(n_{1}00) = \lambda_{2}f(n_{1}00) = rf((n_{1}+1)10) + q^{2}f((n_{1} - 1)00)$\\
	\\
	$(A_{1}f)(n_{1}n_{1}0) = \lambda_{1}f(n_{1}n_{1}0) = rf((n_{1}+1)n_{1}0) + q^{2}f((n_{1} - 1)(n_{1} - 1)0)$\\
	$(A_{2}f)(n_{1}n_{1}0) = \lambda_{2}f(n_{1}n_{1}0) = qrf(n_{1}(n_{1} - 1) 0) + f((n_{1}+1)(n_{1}+1)0)$\\
	\\
	$(A_{1}f)(n_{1}10) = \lambda_{1}f(n_{1}10) = f((n_{1}+1)10) + qf(n_{1}20) + q^{2}f((n_{1} - 1)00)$\\
	$(A_{2}f)(n_{1}10) = \lambda_{2}f(n_{1}10) = f((n_{1}+1)20) + q^{2}f((n_{1} - 1)10) +qf(n_{1}00)$\\
	\\
	$(A_{1}f)(n_{1}(n_{1}-1)0) = \lambda_{1}f(n_{1}(n_{1}-1)0) = f((n_{1}+1)(n_{1}-1)0) + qf(n_{1}n_{1}0) + q^{2}f((n_{1}-1)(n_{1}-2)0)$\\
	$(A_{2}f)(n_{1}(n_{1}-1)0) = \lambda_{2}f(n_{1}(n_{1}-1)0) = f((n_{1}+1)n_{1}0) + q^{2}f((n_{1}-1)(n_{1}-1)0)+ \\ qf(n_{1}(n_{1}-2)0)$\\
	\\
	$(A_{1}f)(n_{1}n_{2}0) = \lambda_{1}f(n_{1}n_{2}0) = f((n_{1}+1)n_{2}0) + qf(n_{1}(n_{2}+1)0) +\\ q^{2}f((n_{1}-1)(n_{2}-1)0)$\\
	$(A_{2}f)(n_{1}n_{2}0) = \lambda_{2}f(n_{1}n_{2}0) = f((n_{1}+1)(n_{2}+1)0) + qf(n_{1}(n_{2}-1)0) +\\ q^{2}f((n_{1}-1)n_{2}0)$\\
	\\
	By the above operators, the function $f$ can be defined:\\
	$f(000) = 1$,\\
	$f(100) = \frac{ \lambda_{1}f(000)}{t}$,\\
	$f(110) = \frac{ \lambda_{2}f(000)}{t}$,\\
	$f(200) = \lambda_{1}f(100) - rqf(110)$,\\
	$f(210) = \frac{ \lambda_{2}f(100) - q^{2}f(000))}{r}$,\\
	$f(220) = \lambda_{2}f(110) - qrf(100)$,\\
	$f(n_{1}00) = \lambda_{1}f((n_{1}-1)00) - rqf((n_{1} - 1)10)$, $n_{1} \ge 3$,\\
	$f(n_{1}n_{1}0) = \lambda_{2}((n_{1}-1)(n_{1}-1)0) - qrf((n_{1}-1)(n_{2}-2)0)$, $n_{1} \ge 3$,\\
	\\
	$f(n_{1}10) = \frac{\lambda_{2}f((n_{1}-1)00) - q^{2}f((n_{1}-2)00)}{r}$,\\
	$f(n_{1}10) = \lambda_{1}f((n_{1}-1)10) - qf((n_{1}-1)20) - q^{2}f((n_{1}-2)00)$, $n_{1} \ge 3$,\\
	\\
	$f(n_{1}(n_{1}-1)0) = \frac{\lambda_{1}f((n_{1}-1)(n_{1}-1)0) - q^{2}f((n_{1}-2)(n_{1}-2)0)}{r}$,\\
	$f(n_{1}(n_{1}-1)0) = \lambda_{2}f((n_{1}-1)(n_{1}-2)0) - qf((n_{1}-1)(n_{1}-3)0) - q^{2}f((n_{1}-2)(n_{1}-2)0)$, $n_{1} \ge 3$,\\
	\\
	$f(n_{1}n_{2}0) = \lambda_{2}f((n_{1}-1)(n_{2}-1)0) - qf((n_{1}-1)(n_{2}-2)0) - q^{2}f((n_{1}-2)(n_{2} - 1)0)$\\
	$f(n_{1}n_{2}0) = \lambda_{1}f((n_{1}-1)n_{2}0) - qf((n_{1}-1)(n_{2}+1)0) - q^{2}f((n_{1}-2)(n_{2}-1)0)$, $n_{1} - n_{2} \ge 2$, $n_{2} \ge 2$.\\
	\\
	The aim is to define a function $f$, which belongs to $L^{2}(T)$, i.e., \\
	$\sum_{u \in T}w(u)||f(u)||^{2} < \infty $.\\
	In the case $d=3$, $A_{1}$ and $A_{2}$ are adjoint operators and their eigenvalues are adjoint too.
	The function $f : T \to C$ is:\\
	$f( \bar{0}) = 1$\\
	$f(100) = \frac{ \lambda_{1}}{t}$\\
	$f(110) = \frac{ \lambda_{2}}{t}$\\
	$f(200) = \frac{ \lambda_{1}^{2} - qr \lambda_{2}}{t}$\\
	$f(210) = \frac{ \lambda_{1} \lambda_{2} - q^{2}t}{tr}$\\
	$f(220) = \frac{ \lambda_{2}^{2} - qr \lambda_{1}}{t}$\\
	$f(300) = \frac{ \lambda_{1}^{3} - q(r+1) \lambda_{1} \lambda_{2} + q^{3}t}{t}$\\
	$f(310) = \frac{ \lambda_{2} \lambda_{1}^{2} - qr \lambda_{2}^{2} - q^{2} \lambda_{1}}{rt}$\\
	$f(320) = \frac{ \lambda_{1} \lambda_{2}^{2} - qr \lambda_{1}^{2} - \lambda_{2}q^{2}}{rt}$\\
	$f(330) = \frac{ \lambda_{2}^{3} - q \lambda_{1} \lambda_{2}(r+1) +q^{3}t}{t}$\\
	$f(400) = \frac{ \lambda_{1}^{4} - q \lambda_{2} \lambda_{1}^{2}(r+2) +q^{2}r \lambda_{2}^{2} + q^{3} \lambda_{1}(t+1)}{t}$\\
	$f(410) = \frac{ \lambda_{2} \lambda_{1}^{3} - q \lambda_{1} \lambda_{2}^{2}(r+1) +q^{3} \lambda_{2}(1+r^{2}) - q^{2} \lambda_{1}^{2}}{rt}$\\
	$f(420) = \frac{ \lambda_{1}^{2} \lambda_{2}^{2} - qr( \lambda_{1}^{3} + \lambda_{2}^{3}) + \lambda_{1} \lambda_{2}q^{3}(r+2) - q^{5}t}{rt}$\\
	$f(430) = \frac{\lambda_{1} \lambda_{2}^{3} - q \lambda_{1}^{2} \lambda_{2}(r + 1) - \lambda_{2}^{2}q
		^{2} + \lambda_{1}q^{3}(t+r)}{rt}$\\
	$f(440) = \frac{ \lambda_{2}^{4} - q \lambda_{1} \lambda_{2}^{2}(r+2) + \lambda_{2}q^{3}(t+1) +q^{2}r \lambda_{1}^{2}}{t}$\\
	$f(500) = \frac{ \lambda_{1}^{5} - q \lambda_{1}^{3} \lambda_{2}(r+3) +  \lambda_{1} \lambda_{2}^{2}q^{2}(2r+1)+ \lambda_{1}^{2}q^{3}(t+2) - q^{4} \lambda_{2}(r^{2}+1)}{t}$\\
	$f(510) = \frac{ \lambda_{2} \lambda_{1}^{4} - q \lambda_{1}^{2} \lambda_{2}^{2}(r+2) + q^{3} \lambda_{1} \lambda_{2}(t+r+2) - q^{2} \lambda_{1}^{3} + q^{2}r \lambda_{2}^{3} - q^{5}t}{rt}$\\
	$f(520) = \frac{ \lambda_{1}^{3} \lambda_{2}^{2} - q \lambda_{1} \lambda_{2}^{3}(r+1) + q^{2} \lambda_{2} \lambda_{1}^{2}(t+3q) - qr \lambda_{1}^{4} + q^{3} \lambda_{2}^{2}(r+1) - q^{4} \lambda_{1}(rt+q)}{rt}$\\
	$f(530) = \frac{ \lambda_{1}^{2} \lambda_{2}^{3} - \lambda_{2} \lambda_{1}^{3}q(q+2) + q^{2} \lambda_{1} \lambda_{2}^{2}(q^{2}+4q+1) - qr \lambda_{2}^{4} + q_{3} \lambda_{1}^{2}(q+2) - \lambda_{2}q^{4}(q^{3}+2q^{2}+3q+1)}{rt}$\\
	$f(540) = \frac{ \lambda_{1} \lambda_{2}^{4}- \lambda_{1}^{2} \lambda_{2}^{2}q(r+2) + \lambda_{1} \lambda_{2}q^{3}(t+r+2) - q^{2} \lambda_{2}^{3} + q^{2}r \lambda_{1}^{3} - q^{5}t}{rt}$\\
	$f(550) = \frac{ \lambda_{2}^{5} - q \lambda_{1} \lambda_{2}^{3}(r+3) + q^{2} \lambda_{1}^{2} \lambda_{2}(2r+1) + \lambda_{2}^{2}q^{3}(t+2) - \lambda_{1} q^{4}(t+r)}{t}$\end{exmpl}
\begin{cor}
	For every complex values $\lambda_{1}$ and $\lambda_{2}$, there
	is a simultaneous eigenvector, which is a function from $T$ to
	$\mathbb{C}$. Therefore, in the action of the Hecke algebra on
	unrestricted space of functions $\C^T$, the simultaneous spectrum
	is $\C^2$.
\end{cor}
For example:
If $\lambda_{1}$ and $\lambda_{2}$ are adjoint, we can choose $\lambda_{1} = \lambda_{2} = t$ and get $f( \bar{n}) \equiv 1$, $f \in L^{2}(T)$. If we choose $\lambda_{1} = \rho^{i}t$, $\lambda_{2} = \rho^{-i}t$ ($\rho^{ki}$ is a root of unity of order $3$), then $f( \bar{n}) = \rho^{ki}$, $k = 1,-1,0$. In the second case, we get a coloring of the diagram of $T$:
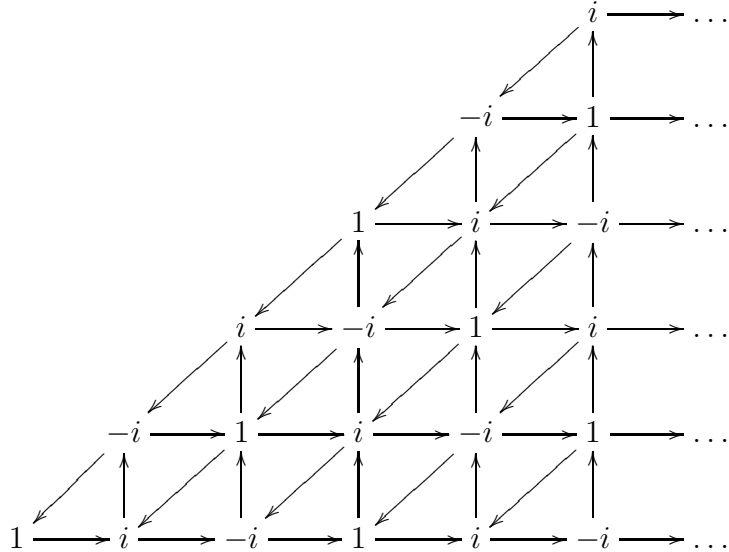
\begin{displaymath}
	\xymatrix{
		& & & & & i \ar[dl] \ar[r] & \dots\\
		& & & & -i \ar[dl] \ar[r] & 1 \ar[u] \ar[dl] \ar[r] & \dots\\
		& & & 1 \ar[dl] \ar[r] & i \ar[u] \ar[r] \ar[dl] & -i \ar[u] \ar[dl] \ar[r] & \dots \\
		& & i \ar[dl] \ar[r] & -i \ar[u] \ar[r] \ar[dl] & 1 \ar[u] \ar[r] \ar[dl] & i \ar[u] \ar[dl] \ar[r] & \dots \\
		& -i \ar[dl] \ar[r] & 1 \ar[u] \ar[r] \ar[dl] & i \ar[u] \ar[r] \ar[dl] & -i \ar[u] \ar[r] \ar[dl] & 1 \ar[u] \ar[dl] \ar[r] & \dots \\
		1 \ar[r] & i \ar[u] \ar[r] & -i \ar[u] \ar[r] & 1 \ar[u] \ar[r] & i \ar[u] \ar[r] & -i \ar[u] \ar[r] & \dots }
\end{displaymath}

The function in the above case is $f(n_{1}n_{2}0)=
\rho^{(n_{1}+n_{2})i}$; moreover, $f \in L^{2}(T)$.


\end{document}

\end{thebibliography}
\end{document}